\documentclass{amsart}
\usepackage{xcolor}
\usepackage{tikz-cd}
\usetikzlibrary{decorations.markings}
\usepackage[english]{babel}
\usepackage{alphabeta}  
\usepackage{amssymb}
\usepackage{amsmath}
\usepackage{bm} 
\usepackage[all]{xy} 
\usepackage{mathrsfs} 

\usepackage{a4wide}

\usepackage{float}	 
\usepackage{graphicx}
\usepackage{chngcntr} 
\usepackage{caption} 

\usepackage{enumerate}

\usepackage[colorlinks,linkcolor={blue},citecolor={blue},urlcolor={purple},]{hyperref}

\usepackage[colorinlistoftodos,prependcaption,textsize=tiny]{todonotes}

\usepackage{easyReview}

\usepackage{mathtools}
\usepackage{physics} 

\usepackage{amsthm}
\usepackage{cleveref}   

\makeatletter
\def\@tvsp{\mathchoice{{}\mkern-4.5mu}{{}\mkern-4.5mu}{{}\mkern-2.5mu}{}}
\def\ltrivert{\left|\@tvsp\left|\@tvsp\left|}
\def\rtrivert{\right|\@tvsp\right|\@tvsp\right|}
\makeatother

\newcommand{\cl}[1]{\overline{#1}}

\renewcommand{\epsilon}{\varepsilon}
\renewcommand{\exp}[1]{e^{#1}}

\renewcommand{\subset}{\subseteq}
\renewcommand{\supset}{\supseteq}
\renewcommand{\1}{{\bf 1}}

\newcommand{\C}{\ensuremath{\mathbb C}}

\newcommand{\R}{\mathbb{R}}

\newcommand{\N}{\mathbb{N}}

\newcommand{\F}{\mathscr F}
\newcommand{\E}{\mathbb E}

\renewcommand{\P}{\mathbb P}
\renewcommand{\O}{\mathcal O}
\renewcommand{\L}{\mathcal L}

\newcommand{\eps}{\varepsilon}
\renewcommand{\tilde}{\widetilde}

\DeclareSymbolFont{matha}{OML}{txmi}{m}{it}
\DeclareMathSymbol{\varv}{\mathord}{matha}{118}

\theoremstyle{plain}

\newtheorem{theorem}{Theorem}[section]
\newtheorem{lemma}[theorem]{Lemma}
\newtheorem{proposition}[theorem]{Proposition}
\newtheorem{corollary}[theorem]{Corollary}

\theoremstyle{definition}
\newtheorem{definition}[theorem]{Definition} 

\newtheorem{assumption}[theorem]{Assumption}

\theoremstyle{remark}

\theoremstyle{remark}
\newtheorem{remark}[theorem]{Remark}

\numberwithin{equation}{section}

\newcommand{\coloneq}{\coloneqq}

\newcommand{\wt}{\widetilde}
\newcommand{\SMR}{{\rm SMR}}
\newcommand{\DSMR}{{\rm DSMR}}

	\allowdisplaybreaks

\title{Fully discrete stochastic maximal regularity and $H^\infty$-calculus for second-order elliptic operators}

\begin{document}

\pagestyle{plain}
 \author{Foivos Evangelopoulos-Ntemiris}
 \address{Delft Institute of Applied Mathematics\\
 Delft University of Technology \\ P.O. Box 5031\\ 2600 GA Delft\\The
 Netherlands} \email{
 F.A.Evangelopoulos-Ntemiris@tudelft.nl and 
 foivosevangelopoulos@gmail.com}

\thanks{The author has received funding from the VICI subsidy VI.C.212.027 of the Dutch Research Council (NWO)}

\begin{abstract} This paper establishes the fully discrete stochastic maximal $L^p$-regularity and the accompanying sharp maximal estimate for numerical approximations of parabolic stochastic partial differential equations.  We consider the spatial finite element discretization $A_h$ of a general second-order elliptic operator $A=-\nabla \cdot a\nabla +b\cdot \nabla +c$  with Dirichlet boundary conditions on a smooth, bounded, convex domain in $\mathbb{R}^3$, coupled with a broad class of temporal schemes, including rational approximations and the exponential Euler method. To obtain these optimal discrete regularity results, we establish a bounded $H^\infty$-calculus for the discrete spatial operator $A_h$, uniformly in the mesh size $h$. As a direct byproduct, we also establish the discrete-in-space stochastic maximal regularity for the corresponding spatial semi-discretizations.

As an application, we establish the pathwise uniform convergence
of a fully discrete approximation of the three-dimensional
stochastic Allen--Cahn equation with variable diffusion coefficients and multiplicative noise.
\end{abstract}



\maketitle
\setcounter{tocdepth}{1}
	\tableofcontents

\section{Introduction}
Stochastic maximal $L^p$-regularity (SMR) techniques play a central role in the theory of stochastic evolution equations of parabolic type. At the core of this theory is the analysis of linear stochastic partial differential equations (SPDEs) of the form 
\begin{equation}\label{eq:SPDE intro}
 \begin{cases}
 du (t) + Au(t) \, dt =  g(t) \,dW(t), \quad t \in (0,T),
 \\
 u(0) =0.
 \end{cases}
\end{equation}
Here, $A$ represents a differential operator subject to appropriate boundary conditions, $W$ denotes a cylindrical Brownian motion on a probability space $\Omega$, $g$ is a given source term and $T\in(0,\infty]$. Broadly speaking, SMR refers to the property that the solution $u$ gains exactly as much spatial and temporal regularity as the governing operator $A$ and the driving noise $g$ allow. In the continuous-in-time setting,  this typically yields an optimal \textit{a priori} estimate that bounds the solution $u$ in $L^p(\Omega \times (0,T); D(A^{1/2}))$ in terms of the norm of the source term $g$. In applications, SMR is usually accompanied by a sharp \textit{maximal estimate}, which bounds the $L^p(\Omega)$-norm of the pathwise supremum of the solution in the real interpolation space $(X_0,X_1)_{\frac 12 -\frac 1p,p}$.

Although maximal regularity is inherently a linear concept, it becomes a powerful tool for analyzing \textit{nonlinear} problems through linearization techniques. In particular, it allows us to establish local well-posedness and regularity results, and formulate sharp blow-up criteria for global existence for a wide class of nonlinear equations (see the monographs \cite{analysis_volume_3, PrussSim} and the surveys \cite{agresti2025nonlinear, wilke2023linear}, as well as the references therein). A standard sufficient condition for establishing maximal regularity for \eqref{eq:SPDE intro} is that the underlying linear operator $A$ admits a bounded $H^\infty$-calculus of angle less than $\pi/2$; see  
\cite{Agresti_Veraar_critical_1, agresti2025nonlinear, NVWSMR, NVW_2012_survey}.

In recent years, applications of SMR to numerical analysis, particularly concerning stability estimates and convergence rates for (nonlinear) SPDEs, have begun to emerge. To approximate the solution of \eqref{eq:SPDE intro} numerically via a spatial semi-discretization, one introduces a finite element space $S_h$ of piecewise linear elements. This yields a continuous-in-time and discrete-in-space equation of the form
\begin{equation}\label{eq:spatial_scheme_intro}
\begin{cases}
    du_h(t) + A_h u_h(t) \, dt =  P_h g(t) \, dW(t), \quad t \in (0,T),
    \\
    u_h(0) = 0,
\end{cases}
\end{equation}
where $A_h$ is the discretization of $A$ and $P_h$ is the $L^2$ projection onto $S_h$. Li and Zhou~\cite{LiZhou2026} showed that the discrete Dirichlet Laplacian $A_h=-\Delta_h$ on a smooth, bounded, convex domain $\O \subset \R^3$ admits a bounded $H^\infty$-calculus of angle zero, with a constant that is independent of the spatial mesh size $h$. As a consequence, the discrete Dirichlet Laplacian $-\Delta_h$ possesses \textit{discrete-in-space} SMR, which is accompanied by a sharp maximal estimate. Using this property, the authors proved the convergence of a spatial semi-discretization for a linear stochastic heat equation. Zhou and Li~\cite{Zhou_Li_2025_allen_cahn} subsequently applied these tools to establish the convergence of a spatial semi-discretization for the three-dimensional stochastic Allen--Cahn equation.

Parallel advances have been made regarding temporal discretizations of \eqref{eq:SPDE intro} of the form 
\begin{equation}\label{eq:temporal scheme intro}
\begin{cases}
U_{n+1} \coloneq r(\tau A) U_{n} + r(\tau A) \int_{t_n}^{t_{n+1}} g(s) \, d W(s), \quad n=0,\dots,N-1,
\\
U_{0}=0,
\end{cases}
\end{equation}
where $r(\cdot)$ is either the exponential function $r(z) \coloneq e^{-z}$ or a consistent and $A$-stable rational function that satisfies $r(\infty) = 0$. For an operator $A$ admitting a bounded $H^\infty$-calculus, Li and Xie~\cite{Li_Xie_timeDSMR} proved that the implicit Euler scheme, given by \eqref{eq:temporal scheme intro} with $r(z) = (1+z)^{-1}$, exhibits \textit{discrete-in-time} SMR. However, their accompanying maximal estimate was suboptimal, suffering from an arbitrarily small $\epsilon$-loss in spatial regularity. 
This framework was subsequently generalized by Veraar and the author in \cite{DSMR}, demonstrating that discrete-in-time SMR is equivalent to its continuous-in-time counterpart for any function $r(\cdot)$ as above, notably without requiring the auxiliary assumption of a bounded $H^\infty$-calculus. Moreover, when $A$ does admit a bounded $H^\infty$-calculus, this framework eliminates the aforementioned $\epsilon$-loss, yielding the sharp maximal estimate.

In this paper, we investigate the \textit{fully discrete} SMR and the accompanying sharp maximal estimate of the full discretization of \eqref{eq:SPDE intro}, namely
\begin{equation}\label{eq:fully discrete scheme intro}
    \begin{cases}
		U^h_{n+1} \coloneq  r(\tau A_h)U^h_{n} +  r(\tau A_h) \int_{t_{n}}^{t_{n+1}} P_h g(s) \, dW(s), \quad n=0,\dots,N-1,
		\\
		U^h_0\coloneq 0,
	\end{cases}
\end{equation}
for a second-order elliptic operator of the form $A=-\nabla \cdot a\nabla + b \cdot \nabla + c$ with Dirichlet boundary conditions on a smooth, bounded, convex domain $\O\subset \R^3$. In the presence of lower-order terms, the results below concern these equations and schemes with $A$ and $A_h$ replaced by $\lambda_0+A$ and $\lambda_0+A_h$, respectively, for a sufficiently large shift $\lambda_0\ge0$ independent of $h$ and $\tau$.
Furthermore, because our analysis relies on establishing a bounded $H^\infty$-calculus for the discrete operator $\lambda_0+ A_h$, uniformly in the mesh size $h$, we simultaneously obtain the discrete-in-space stochastic maximal regularity of \eqref{eq:spatial_scheme_intro} as a direct corollary.
While we focus on the Dirichlet case, we expect that the results can be extended to other boundary conditions, such as the Neumann case, using similar techniques.

The fully discrete SMR of \eqref{eq:fully discrete scheme intro} was recently proved by Li and Zhou~\cite{LiZhou2025_fully_discrete_AC} for the specific case of the discrete Dirichlet Laplacian ($A_h=-\Delta_h$) using the implicit Euler scheme ($r(z)=(1+z)^{-1}$), but with a suboptimal maximal estimate. They subsequently utilized this result to establish the convergence of a fully discrete implicit Euler discretization scheme of the three-dimensional stochastic Allen--Cahn equation. By generalizing both the underlying spatial operator and the temporal approximation scheme, and providing a sharp maximal estimate, the framework developed here is expected to pave the way for establishing sharp stability estimates and convergence rates for a much broader class of second-order SPDEs.

\subsection{Main results} 
We formulate some of the main results of the paper. For the sake of clarity, we restrict our focus to operators without lower-order terms. Let $\O$ be a bounded, convex domain in $\R^3$ of class $C^2$. Let $a= (a_{ij})$ be a real-valued, symmetric, and elliptic matrix with $a_{ij} \in L^\infty (\O)$, and let  $A=-\nabla \cdot a\nabla$ denote the operator associated with the form 
\begin{equation*}
     \mathfrak a (u,v) \coloneq (a\nabla u, \nabla v)_{L^2} , \quad u,v \in H^1_0(\O).
\end{equation*}
Here, $(\cdot,\cdot)_{L^2}$ denotes the inner product in $L^2(\O)$.
We note that the Dirichlet boundary conditions of $A$ are encoded in the domain $H^1_0(\O)$ of the form $\mathfrak a$. 
For $q \in (1,\infty)$ let $X^q_0 \coloneq L^q(\O)$ and let $A_q$ denote the $L^q$ realization of $A$.
Then $A_q$ is an invertible sectorial operator on $X^q_0$ and $A_q$ has a bounded $H^\infty$-calculus of angle zero (see Section \ref{section:The operator Aq}). For $\alpha \in (0,1]$ we define the fractional space $X^q_\alpha=D(A_q^\alpha)$ endowed with the homogeneous graph norm.

Let $S_h\subset H^1_0(\O)$ be the finite element space consisting of piecewise linear elements as defined in Section \ref{section: FEM},
and let $A_h \colon S_h \to S_h$ denote the spatial discretization of $A$ given by 
$$ (A_h u_h,v_h)_{L^2} \coloneq \mathfrak a (u_h,  v_h) , \quad u_h,v_h \in S_h.$$
For $q \in (1,\infty)$ let $X^q_{0,h}$ denote $S_h$ with the $L^q(\O)$-norm. Let $P_h$ denote the $L^2$-orthogonal projection onto $S_h$ given by 
\begin{equation*}
    (u-P_hu,v_h)_{L^2} = 0 , \quad u \in L^2(\O), v_h \in S_h
\end{equation*}
and let $R_h \colon H^1_0(\O) \to S_h$ denote the Ritz projection given by 
$$ \mathfrak a(u-R_hu,v_h)=0, \quad u \in H^1_0(\O),\, v_h \in S_h.$$

Throughout this manuscript, we impose the following assumption: 
\begin{assumption}\label{main assumptions}
Let $q \in (1,\infty)$ and let $q'=q/(q-1)$.
\begin{enumerate}[(i)]
       \item \label{main assumptions: elliptic reg} (Elliptic regularity at $q$ and $q'$) For each $r\in\{q,q'\}$, one has  $X^r_1=W^{2,r}(\O)\cap W^{1,r}_0(\O)$, and there is a constant $C>0$ such that
    \begin{equation}\label{ineq:W^2q regularity}
       C^{-1} \|u\|_{W^{2,r}} \le  \|u\|_{X^r_1} \le C \|u\|_{W^{2,r}}, \quad u \in X^r_1.
    \end{equation}
    \item \label{main assumptions: stability of Rh} ($W^{1,q}$ stability of $R_h$) The Ritz projection $R_h$
is stable on $W^{1,q}(\O)$, that is, there is $C>0$ independent of $h$ such that  
\begin{equation}\label{ineq:assumption Rh stability}
     \|R_h u \|_{W^{1,q}(\O)}  \le C \|u\|_{W^{1,q}(\O)}, \quad u \in W^{1,q}_0(\O).
\end{equation}
\end{enumerate}
\end{assumption}
Some remarks on Assumption \ref{main assumptions} are in order.
\begin{remark} \label{remark:assumptions}
\begin{enumerate}[(1)]
    \item    
    The elliptic regularity estimate \eqref{ineq:W^2q regularity} is satisfied for every $q \in (1,\infty)$ provided the coefficients are sufficiently smooth; specifically, it holds when $a_{ij} \in W^{1,\infty}(\O)$ (see \cite[Theorem 9.14]{GT_pde_book}). For further refinements and sharper results, we refer the reader to \cite{CFF_elliptic_regularity,Vitanza_elliptic_regularity}. Notably, the convexity of the domain is not required in these results. 
\item 
    The stability estimate \eqref{ineq:assumption Rh stability} is a standard assumption and is well studied in the literature. The case $q=2$ holds trivially by the  ellipticity of the coefficient matrix. Moreover, one typically establishes the endpoint case
    \begin{equation}\label{ineq:Stability of Rh on W1infty}
        \|R_h u \|_{W^{1,\infty}(\O)}  \le C \|u\|_{W^{1,\infty}(\O)}, \quad u \in W^{1,\infty}(\O)\cap H^1_0(\O) 
    \end{equation}
    and subsequently applies an interpolation argument to get the range $q \in (2,\infty)$. The range $q\in(1,2)$ is then proved by a duality argument provided the operator $A$ satisfies the elliptic $W^{-1,q}$-regularity estimate
    \begin{equation}\label{ineq:W1q regularity}
        \|u\|_{W^{1,q}(\O)} \le C \|Au\|_{W^{-1,q}(\O)}, \quad u \in W^{1,q}_0(\O).
    \end{equation}
    Indeed, let $q \in (1,2)$ and  $q'=q/(q-1) \in (2,\infty)$. Then by \eqref{ineq:W1q regularity}, the definition of the Ritz projection, and its  $W^{1,q'}$ stability, we obtain  
    \begin{align*}
        \|R_hu\|_{W^{1,q}(\O)} & \le C \sup_{\phi \in W^{1,q'}_0(\O)} \frac{\mathfrak a(R_hu,\phi)}{\|\phi\|_{W^{1,q'}(\O)}} = C \sup_{\phi \in W^{1,q'}_0(\O)} \frac{\mathfrak a(u,R_h \phi)}{\|\phi\|_{W^{1,q'}(\O)}} 
        \\
        & \le C \sup_{\phi \in W^{1,q'}_0(\O)} \frac{\|u\|_{W^{1,q}(\O)} \|R_h \phi\|_{W^{1,q'}(\O)}}{\|\phi\|_{W^{1,q'}(\O)}} \le C \|u\|_{W^{1,q}(\O)}
    \end{align*}
    for $u\in W^{1,q}_0(\O)$. The regularity estimate \eqref{ineq:W1q regularity} is standard and is known to hold for  VMO coefficients; see \cite{Dong_Kim2010}.

    The $W^{1,\infty}$ stability estimate \eqref{ineq:Stability of Rh on W1infty} was first investigated by 
    Rannacher \cite{Rannacher76} and Schatz and Wahlbin \cite{Schatz_wahlbin77} for smooth domains in $\R^d$ ($d\ge 2$). Geissert \cite{Geissert-applications-of-DMR} provided a proof using kernel estimates and even allowed for complex-valued coefficients.  The extension to non-smooth, convex polygonal domains in $\R^2$ was first done by Rannacher and Scott \cite{Rannacher_Scott82}. In their monograph, Brenner and Scott \cite[Theorem 8.1.11]{Brenner_Scott_book} proved the result for more general elements and for non-smooth convex polyhedral domains in $\R^3$ with $W^{1,12/5}$-coefficients under the auxiliary assumption that $A$ possesses elliptic $L^{q}$-regularity for $1<q<\mu$ with $\mu>3$. We note that their arguments can be applied mutatis mutandis to smooth domains for piecewise linear elements as in Section \ref{section: FEM}.

    For the specific case of the Laplacian $A=-\Delta$, the requirement for this extra regularity was subsequently removed in Guzm\'{a}n et al.\,\cite{Guzman_etal_09} using newly established H\"older estimates for the Green function on convex polyhedral domains. Furthermore, for the Laplacian, Diening et al.\,\cite{Diening_et_al24} demonstrated that  $R_h$ is dominated by the maximal operator $\mathcal M$. This implies the stability of $R_h$ in every function space where $\mathcal M$ is bounded, including (weighted) $L^q$ spaces. 

    \item \label{remark:assumptions sufficient condition} Summarizing the above observations, we conclude that Assumption \ref{main assumptions} is satisfied for every $q \in (1,\infty)$ provided the coefficients  $a_{ij} \in W^{1,\infty}(\O)$.
\end{enumerate}
\end{remark}

We now present the first main result of this paper, which is a special case of Theorem \ref{thm:H-calc for Ah},  namely the boundedness of the $H^\infty$-calculus of $A_h$, uniformly in the mesh size $h$. 
\begin{theorem}[Uniformly bounded $H^\infty$-calculus]\label{theorem:main calc theorem intro}
Let $q \in (1,\infty)$ and suppose that Assumption \ref{main assumptions} holds. Then there exists a constant $C>0$, independent of $h$, such that  $A_h$ has a bounded $H^\infty$-calculus on  $X^q_{0,h}$ of angle zero and constant $C$.
\end{theorem}
Theorem \ref{theorem:main calc theorem intro} follows from  Theorem \ref{theorem:Hcalc for Ah} in Section \ref{sec:principal part of Ah}. This abstract result provides the foundation for the fully discrete stochastic maximal regularity results discussed in the introduction. To state these results precisely, we first establish some necessary notation.
For $\alpha \in (0,1]$ we define the discrete fractional  space $X^q_{\alpha,h}$ to be $S_h$ endowed with the $\|A_h^{\alpha}\cdot\|_{L^q(\O)}$-norm. Theorem \ref{theorem:main calc theorem intro} and Lemma \ref{lem:BIP}  imply that $X^q_{\alpha,h} $ coincides with the complex interpolation space $[X^q_{0,h},X^q_{1,h}]_\alpha$ and the norms are equivalent, uniformly in $h$, that is, there is $C>0$ independent of $h$ such that 
\begin{equation*}
    C^{-1} \|u_h\|_{X^q_{\alpha,h}} \le \|u_h\|_{[X^q_{0,h},X^q_{1,h}]_\alpha} \le C \|u_h\|_{X^q_{\alpha,h}}, \quad \alpha \in (0,1), \, q \in (1,\infty).
\end{equation*}

A vital component of this framework is the following characterization of the discrete fractional spaces:
 \begin{equation}\label{Xah norm equiv Xa norm intro}
			X^q_{\alpha,h} = X^q_\alpha \cap S_h \quad \text{ with equivalent norms uniformly in $h$, for every } \alpha \in [0,1/2],
\end{equation}
i.e., there is $C>0$ independent of $h$ such that 
$$
C^{-1} \|u_h\|_{X^q_{\alpha}} \le   \|u_h\|_{X^q_{\alpha,h}} \le C\|u_h\|_{X^q_{\alpha}} , \quad u_h \in S_h,\, \alpha \in [0,1/2].
$$
Such a characterization is essential, as it allows us to identify the abstract discrete spaces $X^q_{1/2,h}$ and the real interpolation space $(X^q_{0,h},X^q_{1,h})_{1/2-1/p,p}$ appearing in the fully discrete SMR estimates with standard Sobolev and Besov spaces. The equivalence \eqref{Xah norm equiv Xa norm intro} follows from Corollary \ref{cor: Xah norm equiv Xa norm} (see also Corollary \ref{cor: Xah norm equiv Xa norm with lower-order terms}).

Let $T\in (0,\infty]$. Let $h>0$ and $g \in L^p_{\mathbb F}(\Omega \times (0,T); \gamma(H, X^q_{0}))$ be fixed but arbitrary.  Let $r(\cdot)$ be either the exponential function $r(z) \coloneq e^{-z}$, or a consistent and $A$-stable rational function that satisfies $r(\infty) = 0$ (see Section \ref{section:SMR} for the definitions).  We say that $\tau>0$ is {\em admissible} if $T/\tau\in \N$ or $T =\infty$. For an admissible $\tau$ let $\pi_\tau \coloneq \{t_n = n\tau \colon 0 \le n \le N \}$ be a uniform partition of $[0,T]$, where $N = T/\tau$ if $T<\infty$, and $N=\infty$ if $T = \infty$.  We define $(U^h_n)_{n=0}^N$ recursively by
\begin{equation} \label{Eq:Definition of approximation scheme fully discrete intro}
	\begin{cases}
		U^h_{n+1} \coloneq  r(\tau A_h)U_{n}^h +  r(\tau A_h) \int_{t_{n}}^{t_{n+1}} P_h g(s) \, dW(s),\quad n=0,\dots,N-1,
		\\
		U^h_0\coloneq 0.
	\end{cases}
	\end{equation}	

We are now ready to present the second main result of this paper:

\begin{theorem}[Fully discrete SMR]
    \label{thm:fully discrete SMR intro}
 Let $q \in [2,\infty)$ and $p \in (2,\infty)$, where we allow $p=2$ if $q=2$. Suppose that Assumption \ref{main assumptions} holds. Then there exists a constant $C>0$, independent of $h$ and $\tau$, such that for every $g \in L^p_{\mathbb F}(\Omega \times (0,T); \gamma(H, X^q_{0}))$, the discrete solution $(U^h_n)_{n=0}^N$ of \eqref{Eq:Definition of approximation scheme fully discrete intro} satisfies
 \begin{equation}\label{ineq:space-time DSMR with discrete trace space intro}
      \begin{aligned}
          \Big( \E \sup_{n=1,\dots, N} \| U^h_n \|^p_{(X^q_{0,h},X^q_{1,h})_{1/2-1/p,p}} \Big )^{1/p}& +  \Big( \E \sum_{n=1}^{N-1} \tau  \|  U^h_n \|^p_{X^q_{1/2,h}} \Big)^{1/p} 
          \\
          & \le C \|P_h g\|_{L^p(\Omega \times (0,T); \gamma(H, X^q_{0,h}))} 
      \end{aligned}
\end{equation}
and
   \begin{equation}\label{ineq:space-time DSMR intro}
         \begin{aligned}
     \Big ( \E \sup_{n=1,\dots, N} \| U^h_n \|^p_{(X^q_0,X^q_1)_{1/2-1/p,p}} \Big)^{1/p}&+  \Big( \E \sum_{n=1}^{N-1} \tau  \|  U^h_n \|^p_{X^q_{1/2}} \Big)^{1/p} 
      \\
      &\le C \|P_h g\|_{L^p(\Omega \times (0,T); \gamma(H, X^q_{0,h}))} .
      \end{aligned}
   \end{equation}
   Furthermore, the right-hand sides of both estimates can be replaced by  $C\|g\|_{L^p(\Omega \times (0,T); \gamma(H, X^q_{0}))}$.  
\end{theorem}

As a byproduct we obtain the following discrete-in-space stochastic maximal regularity:
\begin{theorem}[Discrete-in-space SMR] \label{thm:discrete space SMR intro}
Let $q \in [2,\infty)$ and $p \in (2,\infty)$, where we allow $p=2$ if $q=2$. Suppose that Assumption \ref{main assumptions} holds. Then there exists a constant $C>0$, independent of $h$, such that for every $g \in L^p_{\mathbb F}(\Omega \times (0,T); \gamma(H, X^q_{0}))$, the discrete mild solution  $u_h$ of 
\begin{equation*}
\begin{cases}
    du_h(t) + A_h u_h(t) \, dt =  P_h g(t) \, dW(t), \quad t \in (0,T),
    \\
    u_h(0) = 0,
\end{cases}
\end{equation*}
satisfies
 \begin{equation}\label{ineq:space DSMR with discrete trace space intro}
 \begin{aligned}
        \Big (\E \sup_{0\le t \le T} \|u_h(t)\|^p_{(X^q_{0,h},X^q_{1,h})_{1/2-1/p,p}} \Big )^{1/p} & +  \|  u_h \|_{L^p(\Omega \times (0,T);X^q_{1/2,h})} 
        \\
        &\le C \|P_hg\|_{L^p(\Omega \times (0,T); \gamma(H, X^q_{0,h}))} 
 \end{aligned}
    \end{equation}
    and 
    \begin{equation}\label{ineq:space DSMR intro}
   \begin{aligned}
       \Big ( \E \sup_{0\le t \le T} \|u_h(t)\|^p_{(X^q_{0},X^q_{1})_{1/2-1/p,p}} \Big )^{1/p} &+  \|  u_h \|_{L^p(\Omega \times (0,T);X^q_{1/2})} 
       \\
       &\le C \|P_hg\|_{L^p(\Omega \times (0,T); \gamma(H, X^q_{0,h}))}. 
   \end{aligned}
\end{equation}
Furthermore, the right-hand sides of both estimates can be replaced by  $C\|g\|_{L^p(\Omega \times (0,T); \gamma(H, X^q_{0}))}$.  
\end{theorem}

The proofs of Theorems \ref{thm:fully discrete SMR intro} and \ref{thm:discrete space SMR intro} are presented in Section \ref{sec:fully DSMR} (see Theorems \ref{thm:fully discrete SMR} and \ref{thm:space DSMR}).

\subsection*{Overview} 
In Section \ref{sec:Prel}, we present some essential mathematical background, covering sectorial operators and functional calculus, properties of the elliptic operator $A$, and some standard theory of finite elements. We also review stochastic integration theory and (discrete) stochastic maximal regularity.

Section \ref{sec:calculus proof} is devoted to the first main result of this paper (Theorem \ref{thm:H-calc for Ah}), which establishes a bounded $H^\infty$-calculus for the discrete operator $A_h$, uniformly in the mesh size $h$. 
The proof is structured in two stages. First, in Section \ref{sec:principal part of Ah}, we establish the result for the principal part $A_h^\#$ of $A_h$ (Theorem \ref{theorem:Hcalc for Ah}); this requires analyzing the specific properties of $A_h^\#$ in Section \ref{sec:properties of Rh} before concluding the proof for $A_h^\#$ in Section \ref{sec: proof of calculus for principal part}. Second, in Section \ref{sec:with lower-order}, we prove Theorem \ref{thm:H-calc for Ah} for the full operator $A_h$ via a perturbation argument.

Section \ref{sec:fully DSMR} presents our second main result (Theorem \ref{thm:fully discrete SMR}), establishing the fully discrete stochastic maximal regularity of \eqref{Eq:Definition of approximation scheme fully discrete intro} for a broad class of temporal approximation schemes.

Finally, Section \ref{sec:applications} applies these results to
the three-dimensional stochastic Allen--Cahn equation with variable
diffusion coefficients and multiplicative noise. We prove pathwise uniform convergence of a fully discrete
finite element and implicit Euler approximation scheme.

\subsection*{Notation} Throughout this paper, $C$ denotes a generic positive constant that may change from line to line but remains independent of the mesh size $h$ and the step size $\tau$. For $a,b\in \R$ we will use the standard notation $a\lesssim b$ in case there is a generic constant $C>0$ such that $a\leq Cb$. Moreover, we write $a\eqsim b$ if $a\lesssim b$ and $b \lesssim a$. 

In the interest of brevity, the domain $\O$ is omitted from the notation of function spaces when no confusion is likely to arise; e.g., $L^q$ refers to $L^q(\O)$. The inner product of $L^2$ is denoted by $(\cdot,\cdot)_{L^2}$.
For $q \in (1,\infty)$ we set $X^q_0\coloneq L^q$ and define $X^q_{0,h}$ to be the finite element space $S_h$ endowed with the $L^q$-norm. 

Let $A$ be an elliptic operator in divergence form with Dirichlet boundary conditions, associated with the form $\mathfrak a$. We denote by $A_q$ the $L^q$-realization of $A$. The $L^2$-orthogonal projection onto the finite element space $S_h$ is denoted by $P_h$, and $R_h$ denotes the Ritz projection associated with $A$. Moreover, $A_h$ denotes the spatial finite element discretization of $A$.



  For $\nu \in (0,\pi/2)$, let $\Sigma_{\nu} \coloneq \{z\in \C\setminus\{0\}\colon  |\arg(z)|<\nu\}$ and let $\partial \Sigma_{\nu}$ denote its boundary oriented counterclockwise. For an operator $S$,  $R(\lambda, S)\coloneq (\lambda-S)^{-1}$ denotes the resolvent for $\lambda \in \rho(S)$.
  
  Finally, the function $r\colon \Sigma_{\theta}\to \C$ denotes a rational function or $r(z)=e^{-z}$.

\subsubsection*{Acknowledgments} 
 The author thanks Floris Roodenburg and Mark Veraar for their helpful comments and suggestions.

\section{Preliminaries}\label{sec:Prel}

\subsection{Sectorial operators, functional calculus, and interpolation}

For details on sectorial operators, semigroup theory, and functional calculus, the reader is referred to \cite{zbMATH01354832, Haase:2, analysis_volume_2, analysis_volume_3, kunstmann2004maximal}. We briefly recall some of the key concepts used throughout the paper. For details on interpolation theory the reader is referred to \cite{analysis_volume_1, Tri95}. We will rely on standard results on real and complex interpolation. 

Let $(A,D(A))$ be a closed operator on a Banach space $X$. The operator $A$ is called {\em sectorial} if the domain and the range of $A$ are dense in $X$ and there exists $\nu\in (0,\pi/2)$ such that $\sigma(A)\subseteq \overline{\Sigma_{\nu}}$
and there exists $C>0$ such that
\begin{equation}
\label{eq:sectorialA}
|\lambda|\|R(\lambda,A)\|_{\L(X)}\leq C, \ \ \lambda \in \C\setminus \overline{\Sigma_{\nu}}.
\end{equation}
The {\em angle of sectoriality} $\omega(A)\in[0,\pi)$ is defined as the infimum over all $\nu$ for which a $C$ exists such that \eqref{eq:sectorialA} holds. 

If $\omega(A)<\pi/2$, then $-A$ generates a strongly continuous semigroup $(e^{-tA})_{t\geq 0}$, which extends to a bounded analytic function on $\Sigma_\nu$ for some $\nu \in (0,\pi/2)$ (see \cite[Example 10.1.3]{analysis_volume_2}).

\subsubsection{Functional calculus} 
Let $H^1(\Sigma_{\theta})$ be the Hardy space consisting of all analytic functions $f\colon \Sigma_{\theta}\to \C$ such that 
\[\|f\|_{H^1(\Sigma_{\theta})} \coloneq \sup_{|\phi|<\theta} \int_{0}^\infty |f(s e^{i\phi})| \frac{ds}{s}<\infty.\]
Moreover, $H^\infty(\Sigma_{\theta})$ is the space of bounded analytic functions $f\colon\Sigma_{\theta}\to \C$ equipped with the supremum norm.

If $A$ is a sectorial operator of angle $\omega(A)\in [0,\pi)$ and $f\in H^1(\Sigma_{\theta})$ with $\theta\in (\omega(A),\pi)$ then one can define the bounded operator $f(A)$ by the Dunford integral (contour oriented counterclockwise)
 $$f(A) \coloneq  \frac{1}{2\pi i} \int_{\partial \Sigma_{\nu}} f(z) R(z,A) \,dz,$$
 where $\nu \in (\omega(A),\theta)$ is chosen arbitrarily  (see \cite[Section 10.2]{analysis_volume_2}). Moreover, there is a constant $C=C(\theta,A)$ such that 
 \begin{equation*} 
 	\sup_{t>0}	\|f(tA)\|_{\L(X)} \le C \|f\|_{H^1(\Sigma_\theta)}.
 \end{equation*}

The above $H^1$-calculus is useful and can be used for any sectorial operator. In many important cases one can even prove that for every $f\in H^1(\Sigma_{\theta}) \cap H^\infty(\Sigma_{\theta})$ one has
\begin{align}\label{eq:Hinfty}
 \|f(A)\|_{\L(X)}\leq C\|f\|_{H^\infty(\Sigma_{\theta})}.
\end{align}
In this case we say that $A$ has a {\em bounded $H^\infty(\Sigma_\theta)$-calculus}. One can show that \eqref{eq:Hinfty} uniquely extends to all $f\in H^\infty(\Sigma_{\theta})$. The infimum over all possible $\theta$ is called {\em the angle of the $H^\infty$-calculus}. 

 By now, large classes of sectorial operators $A$ are known to have a bounded $H^\infty$-calculus. One could even say that on $L^q$-spaces the counterexamples are typically only rather academic. A comprehensive list of examples can be found in the notes of \cite[Chapter 10]{analysis_volume_2}.

A general class of operators with a bounded $H^\infty$-calculus, in the case where $X$ is a Hilbert space, is given by the following: all operators $A$ for which 
$-A$ generates a contraction semigroup on $X$. In this case, $\omega(A)$ coincides with the angle of the $H^\infty$-calculus. For details, the reader is referred to \cite[Theorems 10.2.24 and 10.4.21]{analysis_volume_2}.

Another class can be given on $X = L^q$ for $q\in (1, \infty)$: all operators $A$ for which $-A$ generates a positive contraction semigroup on $X$. In this case, one obtains that the angle of the $H^\infty$-calculus is $\leq \pi/2$. Moreover, if $A$ is also sectorial of angle $\omega(A)<\pi/2$, then one obtains that the angle of the $H^\infty$-calculus is $<\pi/2$ as well (although the two angles might differ). Details can be found in \cite[Theorems 10.7.12 and 10.7.13]{analysis_volume_2}.

\subsubsection{Fractional powers}
For a sectorial operator $A$ on a Banach space $X$, the fractional power $A^{z}$ can be defined for any $z\in \C$ via the so-called extended functional calculus (see \cite[Chapter 3]{Haase:2} and \cite[Chapter 15]{analysis_volume_3}). In general, these are again closed, unbounded operators. 

The operator $A$ is said to have {\em bounded imaginary powers (BIP)} if $A^{it}$ extends to a bounded operator on $X$ for any $t\in \R$. In particular, this holds if $A$ has a bounded $H^\infty$-calculus. Indeed, one can apply the calculus to the analytic and bounded function $z\mapsto z^{it}$. An important consequence of BIP is the following identification of the domains of the fractional powers and the complex interpolation spaces.
\begin{lemma}\label{lem:BIP}
Suppose that $A$ is a sectorial operator and that $A$ has BIP. Then for all $\theta\in (0,1)$, \[D(A^{\theta}) = [X, D(A)]_{\theta}\]
with equivalent norms.
\end{lemma}
A proof can be found in \cite[Theorem 6.6.9]{Haase:2}, \cite[Theorem 15.3.9]{analysis_volume_3}, and \cite[1.15.3]{Tri95}.

\subsection{Properties of the operator $A$} \label{section:The operator Aq}  Let $\O$ be a bounded, convex domain in $\R^3$ of class $C^{2}$. 
Let $a_{ij}, b_i,c$ be bounded, measurable, real-valued coefficients, and assume that the matrix $a=(a_{ij})_{i,j=1}^3$ is symmetric and elliptic, namely, $a_{ij}=a_{ji}$ and there is $\eta>0$ such that 
\begin{equation*}
	 \eta |\xi|^2 \le 	\sum_{i,j=1}^3 a_{ij}(x) \xi_i \cl{ \xi_j} \le \eta^{-1} |\xi|^2, \quad \text{a.a. } x \in \O,\, \xi \in \C^3.
\end{equation*}
Let $A = -\nabla \cdot a\nabla + b \cdot \nabla +c$ denote the operator associated with the form 
\begin{equation}\label{def:definition of form}
      \mathfrak a (u,v) \coloneq (a\nabla u, \nabla v)_{L^2} + (b\cdot \nabla u, v)_{L^2}+ (c u,v)_{L^2} , \quad u,v \in H^1_0,
 \end{equation} 
 where $b=(b_i)_{i=1}^3$. Let $A^\#= - \nabla \cdot a \nabla$ denote the principal part of $A$. 
It is well known that the semigroup $(\exp{-tA^\#})_{t\ge0}$ generated by $-A^\#$ on $L^2$ is given by a kernel $K$ that satisfies the kernel estimate 
\begin{equation}\label{ineq:gaussian upper estimate}
    K(t,x,y) \le C t^{-3/2}\exp{-\delta|x-y|^2/t}, \quad t>0, \text{ a.a. } x,y\in \O,
\end{equation}
for some constants $\delta,C>0$; see \cite[Theorem 6.10]{Ouhabaz} for instance. 
By \cite[Proposition 10.2.23]{analysis_volume_2}, $A^\#$ is a sectorial and invertible operator on $ L^2$, admitting a bounded $H^\infty$-calculus of angle zero. Therefore, \cite[Theorem 3.4]{calculus_via_kernel_estimates} and the kernel estimate \eqref{ineq:gaussian upper estimate} imply that, for $q\in (1,\infty)$, the $L^q$-realization $A_q^\#$ of $A^\#$ is a sectorial and invertible operator on $X^q_0\coloneq L^q$ and 
\begin{equation}\label{principal part has calculus}
    A_q^\#  \text{ admits a bounded $H^\infty$-calculus of angle zero on $X^q_0$.}
\end{equation}

Let $A_q$ denote the $L^q$-realization of $A$.
By a standard perturbation argument (cf.\,\cite[Theorem 16.2.7]{analysis_volume_3}), for every $q \in (1,\infty)$ and $\nu \in (0,\pi/2)$ there is $\lambda_0  \ge 0$ large enough such that $\lambda_0+A_q$ is invertible and
\begin{equation}\label{lower-order shift has calculus}
	\lambda_0+A_q  \text{ has a bounded $H^\infty(\Sigma_\nu)$-calculus on $X^q_0$.}    
\end{equation}
We note that 
\begin{equation}\label{eq:domains equal with lower-order terms}
	D((\lambda_0+A_q)^\alpha )  =  D((A_q^\#)^\alpha ) \, \text{ with equivalent norms for every } \alpha \in (0,1]. 
\end{equation}
Indeed, the case $\alpha=1$ follows from \cite[Theorem 16.2.3]{analysis_volume_3} and the case $\alpha \in (0,1)$ by complex interpolation and Lemma \ref{lem:BIP}.
For $\alpha \in (0,1]$ we define the fractional spaces
$$X^{q,\#}_\alpha = D((A_q^\#)^\alpha) \quad \text{and} \quad   X^q_\alpha \coloneq  D((\lambda_0+A_q)^\alpha)$$
endowed with their respective homogeneous graph norms, which are equivalent to the inhomogeneous ones since $A_q^\#$ and $\lambda_0+A_q$ are invertible. Note that by \eqref{eq:domains equal with lower-order terms},
\begin{equation}\label{eq:fraction space equals both domains}
	X^q_\alpha = X^{q,\#}_\alpha\quad \text{with equivalent norms for every } \alpha \in (0,1].
\end{equation}
By Lemma \ref{lem:BIP}, for each $\alpha \in (0,1)$, the fractional space $X^{q,\#}_\alpha$ coincides with the complex interpolation space $[X^{q,\#}_0,X^{q,\#}_1]_\alpha$, with equivalent norms. 
By the sectoriality and invertibility of $A_q^\#$, for every $\nu \in (0,\pi/2)$, there is $C>0$ such that 
\begin{equation}\label{sectoriality of Aq}
\begin{aligned}
     \| R(\lambda,A_q^\#)\|_{\L(X^q_0,X^{q,\#}_1)}&\le C , \quad \lambda \in \C \setminus \Sigma_\nu
     \\
    \| R(\lambda, A_q^\#)\|_{\L(X^q_0)} & \le \frac{C}{1+|\lambda|}, \quad \lambda \in \C \setminus \Sigma_\nu.
\end{aligned}
\end{equation}
Therefore, \eqref{sectoriality of Aq} and complex interpolation imply
\begin{equation}\label{sectoriality of Aq interpolated}
    \| R(\lambda,A_q^\#)\|_{\L(X^{q,\#}_\beta,X^{q,\#}_{\gamma})} \le C (1+|\lambda|)^{-1+\gamma-\beta} , \quad \lambda \in \mathbb C \setminus \Sigma_\nu, \,  0 \le \beta\le \gamma \le 1 .
\end{equation}

Suppose now that Assumption \ref{main assumptions} holds for the principal part $A^\#$ of $A$. Thus, $X^{q,\#}_1 =W^{2,q}\cap W^{1,q}_0$ and the norm equivalence \eqref{ineq:W^2q regularity} is satisfied, and the Ritz projection $R_h^\#$ associated with $A^\#$ is stable on $W^{1,q}$. It then follows from \eqref{eq:fraction space equals both domains} and by complex interpolation (cf.\,\cite[Theorem 4.1]{Seeley_72}) that 
\begin{equation}\label{X12 characterization}
	X^q_{1/2} = X^{q,\#}_{1/2} = W^{1,q}_0 \quad \text{with equivalent norms}.
\end{equation}

\subsection{Finite elements}\label{section: FEM}
 Let $\O$ be a bounded, convex domain in $\R^3$ of class $C^2$. Let $\mathcal K_h$ be a conforming and quasi-uniform triangulation of $\O$ consisting of three-dimensional simplices, where $h$ denotes the maximum diameter of the elements in $\mathcal K_h$; see \cite[Chapter 3.1]{vexler} for the definitions. Let $\O_h$ denote the polyhedral domain induced by $\mathcal K_h$. We assume that every vertex on $\partial \O_h$ also lies on $\partial \O$. We define the finite element space
$$ S_h \coloneq \{ u_h \in C(\cl \O)  \colon u_h =0 \text{ on } \cl \O \setminus \O_h \text{ and $u_h$ is linear on $K$ for each } K \in \mathcal K_h \} .$$
It is standard that, for $1\le q\le \infty$, the following local inverse estimate holds 
\begin{equation}\label{ineq:local inverse estimate}
    		\|\nabla u_h\|_{L^q(K)} \le C h^{-1} \|u_h\|_{L^q(K)}  , \quad u_h \in S_h, \,  K \in \mathcal K_h,
    \end{equation} 
see, e.g., \cite[Lemma 12.1]{ea1},  \cite[Assumption A.2]{Schatz-Lars} or \cite[Theorem 3.2.6]{Ciarlet}. Since $u_h=0$ on $\cl \O \setminus \O_h$, the local inverse estimate \eqref{ineq:local inverse estimate} implies the global inverse estimate
    	\begin{equation}\label{ineq:inverse estimate for W spaces}
    		\|\nabla u_h\|_{L^q} \le C h^{-1} \|u_h\|_{L^q} , \quad u_h \in S_h.
    	\end{equation}
    	Therefore, by the characterization $X^q_{1/2}=W^{1,q}_0$ given in \eqref{X12 characterization} and complex interpolation, for every  $q \in (1,\infty)$ and  $0\le \alpha\le 1/2$, 
    	\begin{equation}\label{ineq:inverse estimate for Xq spaces}
    		\|u_h\|_{ X^q_\alpha}\le C h^{-2\alpha} \|u_h\|_{ X^q_0}, \quad u_h \in S_h.
    	\end{equation}
        
Let $P_h$ denote the $L^2$-orthogonal projection onto $S_h$ given by 
\begin{equation*}
    (u-P_hu,v_h)_{L^2} = 0 , \quad u \in L^2, v_h \in S_h.
\end{equation*}
We recall some standard properties of $P_h$. 
\begin{lemma} Let $q \in (1,\infty)$. Then, there is 
 a constant $C>0$, independent of $h$, such that the projection $P_h$ satisfies the stability bound
 \begin{equation}\label{stability of Ph on X0}
\|P_h\|_{\L(X^q_0)}   \le C
\end{equation}
and the convergence estimate
\begin{equation}\label{Convergence estimate for I-Ph}
            \|P_hu-u\|_{X^q_0}+ h\|\nabla(P_hu-u)\|_{L^q(\O_h)} \le C h^2\|u\|_{X^q_1},  \quad  u \in X^q_1.
\end{equation}
\end{lemma}
\begin{proof}
The stability bound \eqref{stability of Ph on X0} is well known and can be found in \cite{Douglas} for instance. The convergence estimate \eqref{Convergence estimate for I-Ph} is also well known, but we provide a proof for the reader's convenience. There is a Zhang-type interpolant $\tilde I_h$ such that 
\begin{equation}\label{ineq:Ih estimate}
    	\|\tilde I_h u - u\|_{X^q_0} + h \| \nabla (\tilde I_hu -u)\|_{L^q(\O_h)}   \le C h^2 \|u\|_{X^q_1},\quad u \in X^q_1,
\end{equation}
see \cite[Lemmas 2.3 and 2.1]{Geissert-applications-of-DMR}. 
By the stability of $P_h$ on $X^q_0$ given by \eqref{stability of Ph on X0}, the identity $P_h \tilde I_h u = \tilde I_hu$ and \eqref{ineq:Ih estimate}, we get that for $u \in X^q_1$,
	\begin{equation*}
		\|P_h u -u\|_{X^q_0} \le \|P_h (u -\tilde I_h u)\|_{X^q_0} + \|\tilde I_h u -u\|_{X^q_0} \le C  \|\tilde I_h u -u\|_{X^q_0}  \le C h^2 \|u\|_{X^q_1},
	\end{equation*}
    which, together with the inverse estimate \eqref{ineq:inverse estimate for W spaces}, implies
    \begin{align*}
		\|\nabla(P_h u -u)\|_{L^q(\O_h)} &\le 	\|\nabla(P_h u -\tilde I_hu)\|_{L^q(\O_h)} + \|\nabla(\tilde I_h u -u)\|_{L^q(\O_h)}
		\\
		& \le C h^{-1} \|P_h u -\tilde I_h u\|_{X^q_0} +  \|\nabla(\tilde I_h u -u)\|_{L^q(\O_h)}
		\\
		&\le C h \|u\|_{X^q_1},
	\end{align*}
which finishes the proof.
\end{proof}
We define the discrete space
$$X^q_{0,h} $$
to be  $S_h$  endowed with the $L^q$-norm.

\subsection{Stochastic integration} 
In principle, there is a full analogue of stochastic integration theory in infinite dimensions. However, geometric conditions on the underlying spaces are required. In order to give a satisfactory explanation of stochastic integration in a Banach space setting, we need $\gamma$-radonifying operators $\gamma(H,X)$ where $H$ is a Hilbert space and $X$ is a Banach space. For details we refer to \cite[Chapter 9]{analysis_volume_2}. Specializing to Hilbert spaces $X$, this class of operators reduces to the Hilbert--Schmidt operators $\L_2(H,X)$. Moreover, in the important case $X = L^q(\mathcal{O})$, one can identify $\gamma(H,X)$ with $L^q(\mathcal{O};H)$.

Let $(\Omega,\F, \P)$ denote a probability space with filtration $\mathbb{F}=(\F_t)_{t\geq 0}$.
\begin{definition}
\label{def:Cylindrical_BM}
Let $H$ be a Hilbert space. A bounded linear operator $W\colon L^2(\R_+;H)\rightarrow L^2(\Omega)$ is said to be a {\em cylindrical Brownian motion} in $H$ if the following are satisfied:
\begin{itemize}
\item for all $f\in L^2(\R_+;H)$ the random variable $W(f)$ is centered Gaussian;
\item for all $t\in \R_+$ and $f\in L^2(\R_+;H)$ with support in $[0,t]$, $W(f)$ is $\F_t$-measurable;
\item for all $t\in \R_+$ and $f\in L^2(\R_+;H)$ with support in $[t,\infty)$, $W(f)$ is independent of $\F_t$;
\item for all $f_1,f_2\in L^2(\R_+;H)$ we have $\E(W(f_1)W(f_2))=(f_1,f_2)_{L^2(\R_+;H)}$.
\end{itemize}
\end{definition}
Given $W$, the process $t\mapsto W(\1_{(0,t]} h)$ is a Brownian motion for each $h\in H$.

The way to think about $W$ is that it is given by $t\mapsto \sum_{n\geq 1} W^n(t) h_n$, where
$(W^n)_{n\geq 1}$ are independent standard $\mathbb{F}$-Brownian motions, and $(h_n)_{n\geq 1}$ is an orthonormal basis for $H$. However, since the above series does not define an $H$-valued random variable, the definition is given in a weaker sense.

However, the following convergence property does hold: if $S\in \gamma(H,X)$, then
\begin{equation}\label{eq:convSWH}
S W(t) \coloneq \sum_{k\geq 1} W(\1_{(0,t]} h_k) S h_k,
\end{equation} 
 where the convergence takes place in $L^p(\Omega;X)$ for all $p\in [1, \infty)$.

A process $g \colon \R_+\times\Omega \to \L(H,X)$ is called {\em $H$-strongly progressively measurable} if for all $t\in [0,T]$, $g|_{[0,t]}$ is strongly $\mathcal B([0,t])\otimes \F_t$-measurable (where $\mathcal B$ denotes the Borel $\sigma$-algebra).

For $0\leq a<b\leq T$ and a strongly $\F_a$-measurable $\xi\colon \Omega\to \gamma(H,X)$, the stochastic integral of $\1_{(a,b]} \xi$ is defined by
\begin{equation}
\int_0^{t} \1_{(a,b]} \xi\, d W\coloneq  \xi (W(b\wedge t)-W(a\wedge t)),
\end{equation}
where the series can be shown to be convergent as in \eqref{eq:convSWH} by the independence of $\F_a$ and $W(b\wedge t)-W(a\wedge t)$.

The space $L^p_{\mathbb{F}}(\Omega\times(0,T);\gamma(H,X))$ denotes the subspace of $L^p((0,T)\times\Omega;\gamma(H,X))$ consisting of all strongly progressively measurable processes. It can be shown that this coincides with the closure of the adapted step processes of finite rank (see \cite[Proposition 2.10]{NVW1}).

\subsection{(Discrete) stochastic maximal regularity} \label{section:SMR}
For a detailed treatment of stochastic maximal regularity ($\SMR$) and its applications to stochastic partial differential equations (SPDEs) the reader is referred to \cite{AVstab, agresti2025nonlinear,NVWSMR} and the references therein. Let $X_0$ be a Banach space that is isomorphic to a closed subspace of $L^q(\mathcal S)$ for some $q \in [2,\infty)$ and some $\sigma$-finite measure space $\mathcal S$. Suppose that $X_1$ is another Banach space such that $X_1 \hookrightarrow X_0$ densely and let $A$ be a sectorial operator on $X_0$ of angle $\omega(A)<\pi/2$ and domain $D(A)=X_1$. We assume that $A$ is invertible.

\begin{definition}\label{def:SMR}
    Let $T\in(0,\infty]$ and $p\in[2,\infty)$. The operator $A$ is said to have \textit{stochastic maximal $L^p$-regularity} on $(0,T)$ if there is a constant $C>0$ such that, for every $g \in L^p_{\mathbb{F}}(\Omega\times(0,T);\gamma(H,X_0))$, the mild solution $u$ of 
\begin{equation}\label{eq:Stochastic Cauchy Problem}
    \begin{cases}
 du (t) + Au(t) \, dt =  g(t) \,dW(t), \quad t \in (0,T),
 \\
 u(0) =0,
 \end{cases}
\end{equation}
given by 
\begin{align}\label{eq:mildsol}
u(t) \coloneq \int_0^t e^{-(t-s)A} g(s) \, d W(s), \quad t\in(0,T),
\end{align}
belongs to $D(A^{1/2})$ a.s.\ and satisfies
	\begin{equation} \label{Ineq: SMR definition}
		\|A^{1/2}u \|_{L^p(\Omega \times (0,T);X_0)}  \le C \|g\|_{L^p(\Omega\times(0,T);\gamma(H,X_0))}.
	\end{equation}
	The least admissible constant $C$ is denoted by $C_{\SMR(p,T)}^A$. In case the above holds, we will write $A\in \SMR(p,T)$.
\end{definition}
Some basic properties of stochastic maximal $L^p$-regularity are collected in the following:
\begin{proposition}\label{prop:continuoustime}
Suppose that $A$ has stochastic  maximal $L^p$-regularity on $(0,T)$ with respect to a cylindrical Brownian motion on $H$ with $\dim (H)\geq 1$. Then the following hold:
\begin{enumerate}[(1)]
\item\label{it1:continuoustime} If $T<\infty$ and $\lambda\in \C$, then $A+\lambda\in \SMR(p,T)$;
\item\label{it2:continuoustime} If $T=\infty$ and $\Re(\lambda)\geq 0$, then $A+\lambda\in \SMR(p,\infty)$;
\item\label{it3:continuoustime} If $T<\infty$ and $\lim_{t\to \infty}\|e^{-tA}\|_{\L(X_0)} = 0$, then  $A\in \SMR(p,\infty)$;
\item\label{it4:continuoustime} If $\wt{T}\in (0,\infty)$, then $A\in \SMR(p,\wt{T})$;
\item\label{it5:continuoustime} If $q\in (2, \infty)$, then $A\in \SMR(q,T)$;
\item\label{it6:continuoustime} If $\wt{H}$ is another Hilbert space, then $A\in \SMR(p,T)$ with respect to any cylindrical Brownian motion on $\wt{H}$.
\end{enumerate}
\end{proposition}
\begin{proof}
The permanence properties \eqref{it1:continuoustime} and \eqref{it2:continuoustime} follow from \cite[Proposition 3.8]{AVstab}, \eqref{it3:continuoustime} follows from \cite[Theorem 5.2]{AVstab}, \eqref{it4:continuoustime} follows from \cite[Corollary 5.3]{AVstab} and \eqref{it6:continuoustime} follows from \cite[Theorem 3.9]{AVstab}. The extrapolation property 
\eqref{it5:continuoustime} for $T = \infty$ follows from \cite[Theorem 8.2]{LoVer}. If $T<\infty$, then we can use a simple shift argument. Let $\lambda\geq 0$ be such that $\lim_{t\to \infty}\|e^{-t(\lambda+A)}\|_{\L(X_0)}=0$. By \eqref{it1:continuoustime}, $\lambda+A\in \SMR(p,T)$, and thus $\lambda+A\in \SMR(p,\infty)$ by \eqref{it3:continuoustime}. Hence $\lambda+A\in \SMR(q,\infty)$. By \eqref{it4:continuoustime}, $\lambda+A \in \SMR(q,T)$ and thus by  \eqref{it1:continuoustime} this implies $A\in \SMR(q,T)$.
\end{proof}
A sufficient condition for $\SMR$ is given below.

\begin{theorem}\label{thm:SMR in Lq spaces}
   Let  $T \in (0,\infty]$ and $p \in (2,\infty)$, where we allow $p=2$ if $q=2$. Suppose that $A$ admits a bounded $H^\infty$-calculus of angle less than $\pi/2$. 
    Then $A \in \SMR(p,T)$. 

    Moreover, there is  $C>0$ such that for every $g \in L^p_{\mathbb{F}}(\Omega\times(0,T);\gamma(H,X_0))$,  the mild solution $u$ of \eqref{eq:Stochastic Cauchy Problem}  satisfies the maximal estimate
     \begin{equation}\label{ineq:max estimate}
          \Big ( \E \sup_{0\le t \le T} \|u(t)\|^p_{(X_0,X_{1})_{1/2-1/p,p}} \Big)^{1/p} \le C \|g\|_{L^p(\Omega \times (0,T); \gamma(H, X_{0}))}  .
     \end{equation}
\end{theorem}
\begin{proof}
    The first part follows from \cite[Theorem 1.1]{NVWSMR} (see also \cite[Theorems 7.1 and 7.3]{NVW_2012_survey}). The maximal estimate \eqref{ineq:max estimate} follows from \cite[Theorem 1.2]{NVWSMR} (see also \cite[Theorem 7.16]{AVstab}). 
\end{proof}
Recently, in \cite{DSMR} Veraar and the author introduced a discrete-in-time version of stochastic maximal regularity and established its equivalence with the continuous-in-time counterpart. We summarize these results below. 

Let $r$ be either the exponential function $r(z) \coloneq e^{-z}$, or an $A$-stable rational function  $r\colon \Sigma_{\pi/2}\to \C$ (i.e.\ $|r(z)|\leq 1$ for $z\in \Sigma_{\pi/2}$) that is consistent of order $\ell \ge 1$ (i.e.\  $|r(z) - \exp{-z}|\leq C|z|^{\ell+1}$ for $z\to 0$) and satisfies $r(\infty) = 0$. Typical examples include, but are not limited to, $r(z) = (1+z)^{-1}$ with $\ell=1$ which corresponds to the implicit Euler scheme, and the sub-diagonal Pad\'e rational functions $r_{n,n+1}=P_n/Q_{n+1}$ and $r_{n,n+2}=P_n/Q_{n+2}$ $(n\ge 1)$, where
\begin{align*}
    P_n(z) \coloneq \sum_{j=0}^{n} \frac{(n+m-j)! \, n!}{(n+m)! \, j! \, (n-j)!} (-z)^j, \qquad
    Q_m(z) \coloneq \sum_{j=0}^{m} \frac{(n+m-j)! \, m!}{(n+m)! \, j! \, (m-j)!} z^j,
\end{align*}
with $\ell=2n+1$ and $\ell= 2n+2$, respectively. See \cite[Section 2.2]{DSMR} for more details.

Let $T \in (0,\infty]$. We say that $\tau>0$ is {\em admissible} if $T/\tau\in \N$ or $T =\infty$. For an admissible $\tau$ let $\pi_\tau \coloneq \{t_n = n\tau \colon 0 \le n \le N \}$ be a uniform partition of $[0,T]$, where $N = T/\tau$ if $T<\infty$, and $N=\infty$ if $T = \infty$.  For $g \in L^p_{\mathbb{F}}(\Omega \times (0,T);\gamma(H,X_0))$, we define $(U_n)_{n=0}^N$ recursively by
	\begin{equation} \label{Eq:Definition of approximation scheme}
	\begin{cases}
		U_{n+1} \coloneq  r(\tau A) U_{n} + r(\tau A) \int_{t_n}^{t_{n+1}} g(s)\, dW(s), \quad n=0,\dots, N-1,
		\\
		U_0\coloneq 0.
	\end{cases}
	\end{equation}	
Alternatively, we can write $U_n$ as a discrete stochastic convolution
\begin{equation} \label{Eq: Uniform step discrete stochastic convolution}
	U_n = \sum_{j=0}^{n-1}  r(\tau A)^{n-j} \int_{t_j}^{t_{j+1}} g(s)\, dW(s) , \quad n=1,\dots, N.
\end{equation}
\begin{definition} \label{def:DSMR}
Let $T\in(0,\infty]$ and $p\in[2,\infty)$. The scheme $R\coloneq (r(\tau A))_{\tau>0}$ is said to have \textit{discrete stochastic $\ell^p$-maximal regularity} ($\DSMR$) on $(0,T)$ if there is a constant $C>0$, independent of $\tau$, such that for every $g \in L^p_{\mathbb{F}}(\Omega\times(0,T);\gamma(H,X_0))$,
the approximation scheme $(U_n)_{n=0}^N$ given in \eqref{Eq:Definition of approximation scheme} belongs to $D(A^{1/2})$ a.s.\ and satisfies
	\begin{equation} \label{Ineq:DSMR definition}
		 \Big( \E \sum_{n=1}^{N-1} \tau \| A^{1/2} U_n \|^p_{X_0} \Big)^{1/p} \le C \|g\|_{L^p(\Omega \times (0,T); \gamma(H, X_{0}))}.
	\end{equation}
	The least admissible constant $C$ is denoted by $C_{\DSMR(p,T)}^R$. In case the above holds, we  write $R\in \DSMR(p,T)$.
\end{definition}
The constant $C$ is allowed to depend on $T$, but should be uniform in the step size $\tau$.

\begin{remark}
     We note that the spatial regularity of Definition \ref{def:DSMR} is shifted by 1/2 compared to  \cite[Definition 3.1]{DSMR}. This shift is motivated by our intent to apply this framework to the spatial discretization $A_h$ defined in Section \ref{section: FEM}. For this discrete operator, the discrete-to-continuous  norm equivalence \eqref{Xah norm equiv Xa norm intro} holds for $\alpha \in [0,1/2]$, but fails for any $\alpha>1/2$. As a result, the shifted regularity in Definition \ref{def:DSMR} is more appropriate for studying $A_h$. We also note that this shift is justified  by \cite[Remark 3.3]{DSMR}.
\end{remark}

The connection between $\DSMR$ and $\SMR$ is summarized in the following. 

\begin{theorem}\label{thm:SMR iff DSMR}
    Let $T\in (0,\infty]$ and $p\in [2,\infty)$. Then the following are equivalent:
    \begin{enumerate}[(1)]
\item\label{it:mainequiv1} $A$ has stochastic maximal $L^p$-regularity on $(0,T)$;
\item\label{it:mainequiv2} $R$ has discrete stochastic maximal $\ell^p$-regularity on $(0,T)$.
\end{enumerate}
Moreover, $C_{\SMR(p,T)}^A \leq C_{\DSMR(p,T)}^R\leq K (C_{\SMR(p,T)}^A + 1)$,
where the constant $K$ depends on $p$, $X_0$, the function $r$, and on the sectoriality constant and angle of $A$.
\end{theorem}
In particular, note that Theorem \ref{thm:SMR iff DSMR} implies that Definition \ref{def:DSMR} does not depend on the choice of the function $r$. Consequently, we get the following sufficient condition for $\DSMR$.

\begin{theorem}\label{thm:DSMR in Lq spaces}
  Let  $T \in (0,\infty]$ and $p \in (2,\infty)$, where we allow $p=2$ if $q=2$.   Suppose that $A$ admits a bounded $H^\infty$-calculus on $X_0$ with angle less than $\pi/2$. 
    Then  $R \in \DSMR(p,T)$. 

    Moreover, there is $C>0$, independent of $\tau$, such that for every $g \in L^p_{\mathbb{F}}(\Omega\times(0,T);\gamma(H,X_0))$  the approximation scheme $(U_n)_{n=0}^N$ given in \eqref{Eq:Definition of approximation scheme} satisfies the maximal estimate
     \begin{equation}\label{ineq:discrete max estimate}
     \Big ( \E \sup_{n=1,\dots, N} \| U_{n} \|^p_{(X_0,X_1)_{1/2-1/p,p}} \Big)^{1/p}\le C \|g\|_{L^p(\Omega \times (0,T); \gamma(H, X_0))} .
   \end{equation}
\end{theorem}
\begin{proof}
    The first part follows from Theorem \ref{thm:SMR in Lq spaces}  and Theorem \ref{thm:SMR iff DSMR}.  The discrete maximal estimate \eqref{ineq:discrete max estimate} follows from \cite[Theorem 6.1 and Proposition 6.3]{DSMR} by applying a shift of 1/2 to the spatial regularity. Indeed, it suffices to observe that 
   \begin{equation}\label{isomorphism between real int spaces}
       A^{-1/2} \colon (X_0,X_1)_{1/2-1/p,p} \to (X_0,X_1)_{1-1/p,p} \text{ is an isomorphism.}
   \end{equation}
   The latter can be proved as follows. Since $A$ admits a bounded $H^\infty$-calculus, and in particular has BIP, $A^{-1/2}$ acts as an isomorphism from $X_0$ onto $X_{1/2}\coloneq [X_0,X_1]_{1/2}$ and from $X_{1/2}$ onto $X_1$. Hence, by interpolation, $A^{-1/2}$ is an isomorphism from $(X_0,X_{1/2})_{1-2/p,p} $ onto $ (X_{1/2},X_1)_{1-2/p,p}$. The reiteration theorem (cf.\,\cite[Theorem L.3.1]{analysis_volume_3})
   gives $(X_0,X_{1/2})_{1-2/p,p} =  (X_0,X_1)_{1/2-1/p,p}$ and $ (X_{1/2},X_1)_{1-2/p,p}= (X_0,X_1)_{1-1/p,p}$ with equivalent norms, and thus \eqref{isomorphism between real int spaces} follows. 
\end{proof}

\begin{remark}\label{remark:weights and theta stable}
    For the sake of simplicity, we have restricted our presentation to the setting without temporal weights. However, we note that the general framework developed in \cite{DSMR} is established in a more general weighted setting, accommodating appropriate time weights. Hence, the equivalence between $\SMR$ and $\DSMR$, as well as the discrete maximal estimate, hold in the weighted setting.

    Moreover, we can allow for $A(\theta)$-stable rational functions $r$, that is, $|r(z)|\le 1$ for $z\in \Sigma_\theta$, provided the stability angle $\theta$ is strictly greater than the sectoriality angle of the operator $A$, namely, $\theta \in (\omega(A),\pi/2]$. 
\end{remark}

\section{$H^\infty$-calculus for the discrete operator $A_h$}\label{sec:calculus proof}

The first main result of this paper is established in this section: we show that, up to a shift, the discrete operator $A_h$ admits an $H^\infty$-calculus that is bounded uniformly with respect to the mesh size $h$.

Let $\O$ be a bounded convex domain in $\R^3$ of class $C^2$. We consider the elliptic operator
$A = -\nabla \cdot a\nabla  +  b \cdot \nabla +c$ with Dirichlet boundary conditions. 
The coefficients $a_{ij},b_i,c \in L^\infty$ are real-valued and the matrix $a=(a_{ij})$ is symmetric and elliptic. Let $A_h\colon S_h \to S_h$ denote the discretization of $A$ given by 
$$(A_h u_h,v_h)_{L^2} \coloneq (a\nabla u_h,\nabla v_h)_{L^2} + (b\cdot \nabla u_h,v_h)_{L^2}+(cu_h,v_h)_{L^2}, \quad u_h,v_h \in S_h. $$
We also let $A^\#= - \nabla \cdot a\nabla$ denote the principal part of $A$. As usual, $X^q_{0,h}$ denotes $S_h$ with the $L^q$-norm.

\begin{theorem} 
\label{thm:H-calc for Ah} Let $q\in(1,\infty)$ and $\nu \in(0,\pi/2)$. Suppose that Assumption \ref{main assumptions} holds for $A^\#$. 
 Then there exist a sufficiently large $\lambda_0 \ge0$ and a constant $C>0$, both independent of $h$, such that $\lambda_0+  A_h$ admits a bounded $H^\infty(\Sigma_\nu)$-calculus on $X^q_{0,h}$ with constant $C$. Moreover, $\lambda_0+A_h$ is invertible with  $\| (\lambda_0+A_h)^{-1}\|_{\L(X^q_{0,h})} \le C$.
\end{theorem}

\begin{remark}
    The constant $\lambda_0\ge 0$ in Theorem \ref{thm:H-calc for Ah} in general depends on $q$, the coefficients of $A$, and the angle $\nu$. In the case where $b_i=c=0$ one can take $\lambda_0=0$. 
\end{remark}
The proof of Theorem \ref{thm:H-calc for Ah} is divided into two main steps, which we carry out in the following subsections. We define the perturbation operator $B = b\cdot \nabla + c$ containing the lower-order terms, so that $A = A^\# + B$. Correspondingly, we write $A_h = A_h^\# + B_h$, where $A_h^\#$ and $B_h$ denote the spatial discretizations of $A^\#$ and $B$, respectively.  In Section \ref{sec:principal part of Ah} we prove that the principal part $A_h^\#$ admits a bounded $H^\infty$-calculus of angle zero, uniformly in the mesh size $h$. Subsequently, in Section \ref{sec:with lower-order}, we employ a perturbation argument to extend the calculus to the full operator $A_h$. This perturbation step naturally yields the required shift $\lambda_0 +A_h$ needed to absorb the lower-order terms.

\subsection{$H^\infty$-calculus for the principal part $A_h^\#$} \label{sec:principal part of Ah}
In this section we prove that the principal part $A_h^\#$ admits a bounded $H^\infty$-calculus of angle zero, uniformly in the mesh size $h$. This result is stated in the following theorem.

\begin{theorem} \label{theorem:Hcalc for Ah}
Let $q \in (1,\infty)$ and suppose that Assumption \ref{main assumptions} holds for the principal part $A^\#$ of $A$. Then for every $\nu \in (0,\pi/2)$ there is a constant $C>0$, independent of $h$, such that  
\begin{equation}\label{proof of main thm sufficient condition 1}
    \|f(A_h^\#)P_h\|_{\L(X^q_{0})} \le C \|f\|_{H^\infty(\Sigma_\nu)}, \quad f \in H^\infty(\Sigma_\nu) \cap H^1(\Sigma_\nu). 
\end{equation}
\end{theorem}
As mentioned in the introduction, Theorem \ref{theorem:Hcalc for Ah} allows us to characterize the discrete fractional spaces 
$X^{q,\#}_{\alpha,h}$, which are defined to be $S_h$ with the $\|(A_h^\#)^\alpha \cdot \|_{L^q}$-norm.  Specifically, for every $\alpha \in [0,1/2]$, we have$$X^{q,\#}_{\alpha,h} = X^{q,\#}_\alpha \cap S_h,$$
with equivalent norms, uniformly in $h>0$. This is stated in the following:
\begin{corollary}\label{cor: Xah norm equiv Xa norm}
For every $\alpha \in [0,1]$ there is a constant $C>0$, independent of $h$, such that 
\begin{equation}\label{Ph stable Xqbeta to Xqbetah}
    \|P_h\|_{\L(X^{q,\#}_\alpha , X^{q,\#}_{\alpha,h})} \le C.
\end{equation}
 Moreover, for every $\alpha \in [0,1/2]$ there is $C>0$, independent of $h$, such that 
   \begin{equation}\label{Xah norm equiv Xa norm}
			C^{-1} \|u_h\|_{X^{q,\#}_{\alpha}} \le   \|u_h\|_{X^{q,\#}_{\alpha,h}} \le C\|u_h\|_{X^{q,\#}_{\alpha}} , \quad u_h \in S_h.
		\end{equation} 
\end{corollary}
Furthermore, we are able to characterize certain real interpolation scales of the discrete fractional spaces that appear in the fully discrete stochastic maximal regularity estimates. 
\begin{corollary} \label{cor:discrete vs continuous real int}
	Let $q \in (1,\infty)$, $p \in [1,\infty]$ and $\theta \in (0,1)$. Then there is a constant $C>0$, independent of $h>0$, such that the following hold:
    \begin{enumerate}
        \item [$(\mathrm{i})$] For $0 \le \beta <\alpha \le 1/2$ and $u_h \in S_h$,
	\begin{equation}\label{ineq: equivalence of real interpolation for positive beta}
	   C \|u_h\|_{(X^{q,\#}_{\beta,h},X^{q,\#}_{\alpha ,h})_{\theta,p}} \le \|u_h\|_{(X^{q,\#}_{\beta},X^{q,\#}_{\alpha })_{\theta,p}} \le C^{-1}\|u_h\|_{(X^{q,\#}_{\beta,h},X^{q,\#}_{\alpha ,h})_{\theta,p}}. 
	\end{equation}
        \item[$(\mathrm{ii})$] If $\theta \in [0,1/2)$ then for every $u_h \in S_h$, 
    \begin{equation}\label{ineq: equivalence of real interpolation for trace space}
	    C \|u_h\|_{(X^{q,\#}_{0,h},X^{q,\#}_{1,h})_{\theta,p}} \le \|u_h\|_{(X^{q,\#}_{0},X^{q,\#}_{1})_{\theta,p}} \le C^{-1}\|u_h\|_{(X^{q,\#}_{0,h},X^{q,\#}_{1,h})_{\theta,p}}. 
	\end{equation}
    \end{enumerate}
\end{corollary}

The proof of Theorem \ref{theorem:Hcalc for Ah} is divided into two parts. First, in Section \ref{sec:properties of Rh}, we establish essential properties of the operator $A_h^\#$ and its associated Ritz projection $R_h$. Subsequently, in Section \ref{sec: proof of calculus for principal part}, we utilize these properties to conclude the proof of Theorem \ref{theorem:Hcalc for Ah}. The proofs of Corollaries \ref{cor: Xah norm equiv Xa norm} and \ref{cor:discrete vs continuous real int} are presented at the end of Section \ref{sec: proof of calculus for principal part}.
\subsubsection{Properties of $A_h^\#$ and its Ritz projection} \label{sec:properties of Rh}
We recall that $X^q_0 \coloneq L^q$ and $X^q_{0,h}$ denotes $S_h$ endowed with the $L^q$-norm. Moreover, for $\alpha \in (0,1]$ we let $X^{q,\#}_{\alpha} \coloneq D((A_q^\#)^\alpha)$ endowed with the homogeneous graph norm, and define the discrete space $X^{q,\#}_{1,h} $ to be $S_h$ with the $\|A_h^\# \cdot\|_{L^q}$-norm.

Let $\mathfrak a^\#(u,v) = (a\nabla u, \nabla v)$ be the form associated with $A^\#=-\nabla \cdot a \nabla$ and let 
$R_h^\# \colon H^1_0 \to S_h$ be the Ritz projection associated with $A_h^\#$ given by 
$$ \mathfrak a^\#(u-R_h^\#u,v_h)=0, \quad u \in H^1_0, v_h \in S_h.$$
Note that $R_h^\#$ satisfies the identity 
$$R_h^{\#} = (A_h^\#)^{-1} P_h A_q^\# \text { on }  X^{q,\#}_{1}.$$
We will need the following stability and convergence properties of the Ritz projection:
\begin{lemma}\label{lemma:properties of Rh}
  Let $q \in (1,\infty)$ and suppose that Assumption \ref{main assumptions} holds for the principal part $A^\#$. Then there is a constant $C>0$, independent of $h$, such that the Ritz projection $R_h^\#$ satisfies the stability bound 
    \begin{equation}\label{ineq:stability of Rh on Xq1/2}
    \|R_h^\#\|_{\L(X^{q,\#}_{1/2})} \le C
\end{equation}
and the convergence estimate
\begin{align}
    \|I-R_h^\#\|_{\L(X^{q,\#}_1,X^q_0)} \le C h^2 .\label{Convergence estimate for I-Rh}
\end{align}
Moreover, if $q \in [2,\infty)$, then 
\begin{equation}
        \|I-R_h^\#\|_{\L(X^{q,\#}_{1/2},X^q_0)} \le C h .\label{Convergence estimate for I-Rh lower order}
\end{equation}
\end{lemma}
\begin{proof}
    The stability \eqref{ineq:stability of Rh on Xq1/2} follows from Assumption \ref{main assumptions} and the fractional domain characterization $X^q_{1/2}=W^{1,q}_0$ given by \eqref{X12 characterization}. For $q \in [2,\infty)$, the convergence estimates \eqref{Convergence estimate for I-Rh} and \eqref{Convergence estimate for I-Rh lower order} are proved using the stability bound \eqref{ineq:stability of Rh on Xq1/2} as in \cite[Lemma A.7]{Geissert-applications-of-DMR} (see also \cite[Lemma 4.5]{LiZhou2026}). 
    
    It remains to prove  \eqref{Convergence estimate for I-Rh} for $q \in (1,2)$. By the definition of $R_h^\#$ and the symmetry of the form $\mathfrak a^\#$ it follows that 
    $$ [(I-R_h^\#) (A_q^\#)^{-1}]^*= (I-R_h^\#) (A_{q'}^\#)^{-1},$$ where $1/q+1/q'=1$. Hence,  
$$ \| I-R_h^\#\|_{\L(X^q_1,X^q_0)} = \| (I-R_h^\#) (A_q^\#)^{-1} \|_{\L(X^q_0)} = \| (I-R_h^\#) (A_{q'}^\#)^{-1} \|_{\L(X^{q'}_0)} = \| I-R_h^\#\|_{\L(X^{q'}_1,X^{q'}_0)}  .$$
 This implies  \eqref{Convergence estimate for I-Rh} for $q \in (1,2)$. 
\end{proof}
Using Lemma \ref{lemma:properties of Rh} we can establish the following properties of the projection $P_h$:
\begin{lemma}
    Let $q \in (1,\infty)$ and suppose that Assumption \ref{main assumptions} holds for the principal part $A^\#$. Then there is a constant $C>0$, independent of $h$, such that $P_h$ satisfies the  convergence estimate
\begin{equation}
    \|P_h-R_h^\#\|_{\L(X^{q,\#}_1,X^q_0)} \le C h^2.\label{Convergence estimate for Ph-Rh}
\end{equation}
Moreover, if $q \in [2,\infty)$ then $P_h$ satisfies the stability bound
\begin{equation}\label{ineq:stability of Ph on X12}
    \|P_h\|_{\L(X^{q,\#}_{1/2})} \le C
\end{equation}
and the convergence estimate 
\begin{align}
    \|P_h-R_h^\#\|_{\L(X^{q,\#}_{1/2},X^q_0)}& \le C h .\label{Convergence estimate for Ph-Rh on X1/2}
\end{align}
\end{lemma}
\begin{proof}
The convergence estimate \eqref{Convergence estimate for Ph-Rh} follows by combining the convergence estimates for $P_h$ and $R_h^\#$ given by   \eqref{Convergence estimate for I-Ph} and \eqref{Convergence estimate for I-Rh}, respectively. The estimate \eqref{Convergence estimate for Ph-Rh on X1/2} 
follows from the convergence estimate for $P_h$ given in  \eqref{Convergence estimate for I-Ph} and the stability of $R_h^\#$ given in \eqref{ineq:stability of Rh on Xq1/2}, see \cite[Lemma A.7]{Geissert-applications-of-DMR} for instance.

    The stability \eqref{ineq:stability of Ph on X12} follows from \eqref{ineq:stability of Rh on Xq1/2},  the convergence estimate \eqref{Convergence estimate for Ph-Rh on X1/2}, the inverse estimate \eqref{ineq:inverse estimate for Xq spaces} and the triangle inequality as follows 
    \begin{align*}
        \|P_h u\|_{X^{q,\#}_{1/2}} &\le \|P_h u - R_h^\# u\|_{X^{q,\#}_{1/2}} + \|R_h^\# u\|_{X^{q,\#}_{1/2}} \le C h^{-1}\|P_h u - R_h^\# u\|_{X^q_0} + C\|u\|_{X^{q,\#}_{1/2}} 
        \\
        & \le C\|u\|_{X^{q,\#}_{1/2}}.
    \end{align*}
\end{proof}
 We now collect some properties of the operator $A_h^\#$. 
\begin{lemma} Let $q \in (1,\infty)$. Then the operator $A_h^\#$ is bounded on $X^q_{0,h}$ and there is $C>0$, independent of $h$, such that 
    \begin{equation}\label{ineq: norm of Ah}
        \|A_h^\#\|_{\L(X^q_{0,h})} \le C h^{-2}.
    \end{equation}
    Moreover, if Assumption \ref{main assumptions} holds for the principal part $A^\#$, then there is $C>0$, independent of $h$, such that the following convergence and stability bounds hold:
\begin{align}\label{ineq: perturbation of inverses gamma=1}
    \|(A_h^\#)^{-1} P_h - (A_q^\#)^{-1} \|_{\L(X^q_0)} \le C h^{2}
\\
\label{ineq: stability of inverse of Ah}
    \|(A_h^\#)^{-1} \|_{\L(X^q_{0,h})} \le C.
\end{align}
\end{lemma}
\begin{proof}
    In order to prove \eqref{ineq: norm of Ah}, it suffices to show that 
	\begin{equation}\label{norm of AhPh}
		\|A_h^\# P_h u\|_{X^q_0} \le C h^{-2} \|u\|_{X^q_0}, \quad u \in X^q_0,
	\end{equation}
    where $C>0$ does not depend on $h$. We prove \eqref{norm of AhPh} by a duality argument. Let $1/q+1/q'=1$. For $v \in X^{q'}_0$, the definition of $P_h$, the stability estimate \eqref{stability of Ph on X0} and the inverse estimate \eqref{ineq:inverse estimate for Xq spaces} imply
\begin{align*}
		|(A_h^\# P_h u , v)_{L^2}| & =|(A_h^\# P_h u , P_hv)_{L^2} |
		\le C \| P_h u\|_{W^{1,q}} \|P_hv\|_{W^{1,q'}}  \le C   \| P_h u\|_{X^{q,\#}_{1/2}} \|P_hv\|_{X^{q',\#}_{1/2}} 
        \\
        &\le C (h^{-1} \|P_h u\|_{X^q_0}) (h^{-1} \|P_h v\|_{X^{q'}_0})
        \le C h^{-2} \|u\|_{X^q_0} \|v\|_{X^{q'}_0}. 
\end{align*}
The convergence estimate \eqref{ineq: perturbation of inverses gamma=1} follows from the identity $(A_h^\#)^{-1}P_h-(A_q^\#)^{-1}=(R_h^\#-I)(A_q^\#)^{-1}$ and \eqref{Convergence estimate for I-Rh}. Finally, the stability bound \eqref{ineq: stability of inverse of Ah} follows from  \eqref{ineq: perturbation of inverses gamma=1} and the triangle inequality.
\end{proof}
We now establish the following stability properties of $P_h$, which can also be found in \cite[Lemma 4.7]{Geissert-applications-of-DMR}.
\begin{lemma}
    Let $q \in (1,\infty)$ and suppose that Assumption \ref{main assumptions} holds for the principal part $A^\#$. Then $P_h$ is stable from $ X^q_0$ to $X^q_{0,h}$ and from $ X^{q,\#}_1$ to $X^{q,\#}_{1,h}$, namely, there is $C>0$ independent of $h$ such that 
    \begin{align}
\|P_h\|_{\L(X^{q,\#}_0,X^q_{0,h})} &\le C, \label{ineq:stability of Ph from Xq0 to Xq0h}
\\
\|P_h\|_{\L(X^{q,\#}_1 , X^{q,\#}_{1,h})} &\le C.
\label{Ph stable X1 to X1h} 
\end{align}
\end{lemma}
\begin{proof}

Note that \eqref{ineq:stability of Ph from Xq0 to Xq0h} follows from \eqref{stability of Ph on X0} and the norm identity $\|u_h\|_{X^q_0}=\|u_h\|_{X^q_{0,h}}$. By the triangle inequality, \eqref{Ph stable X1 to X1h} will follow once we prove 
$$\| A_h^\# (P_h-R_h^\#) u \|_{X^q_{0,h}} + \|A_h^\# R_h^\#u\|_{X^q_{0,h}} \le C \|u\|_{X^{q,\#}_1} , \quad u \in X^{q,\#}_1.$$
By the identity $A_h^\# R_h^\# = P_h A_q^\#$ and the stability of $P_h$, we get
$$\| A_h^\# R_h^\# u \|_{X^q_{0,h}} = \| P_h A_q^\# u \|_{X^q_{0,h}} =\| P_h A_q^\# u \|_{X^q_{0}}  \le  C \|A_q^\#u\|_{X^q_0} = C\|u\|_{X^{q,\#}_1},$$
whereas \eqref{norm of AhPh} and the convergence estimate \eqref{Convergence estimate for Ph-Rh} give
$$\| A_h^\# (P_h-R_h^\#) u \|_{X^q_0} \le C h^{-2} \|P_hu-R_h^\# u\|_{X^q_0} \le C \|u\|_{X^{q,\#}_1}.$$
\end{proof}
Finally, we prove that $A_h^\#$ is a sectorial operator on $X^q_{0,h}$ of angle zero, uniformly in the mesh size $h$. 
\begin{lemma}\label{lemma:sectoriality of Ah}
Let $q \in (1,\infty)$ and $\nu \in (0,\pi/2)$, and suppose that Assumption \ref{main assumptions} holds for the principal part $A^\#$. Then there exist constants $C > 0$ and $\delta > 0$, independent of $h$, such that $\sigma(A_h^\#) \subset [\delta, C h^{-2}]$ and 
\begin{equation}\label{ineq:sectoriality of Ah}
\| R(\lambda,A_h^\#) \|_{\L(X^q_{0,h})} \le \frac{C}{1+|\lambda|}, \quad \lambda \in \mathbb C \setminus \Sigma_\nu.
\end{equation}
\end{lemma}
The proof of Lemma \ref{lemma:sectoriality of Ah} follows the arguments presented in \cite{Bakaev-Thomee-maximum-norm}. While the authors in \cite{Bakaev-Thomee-maximum-norm} establish \eqref{ineq:sectoriality of Ah} for $q=\infty$, their approach requires higher regularity of the coefficients (see Remark \ref{remark:end point sectoriality} below). To bypass this, we tailor their arguments to the setting $q \in (1, \infty)$.
\begin{proof}
We begin by showing that the spectrum of $A_h^\#$ is contained in $[\delta, Ch^{-2}]$ for some positive constants $\delta$ and $C$ independent of $h$. Because $X^q_{0,h}$ is finite-dimensional, this spectrum is independent of the parameter $q \in (1, \infty)$. Consequently, it suffices to restrict our attention to the case $q=2$. Note that $A_h^\#$ is a positive and symmetric operator on $X^2_{0,h}$ with its smallest eigenvalue bounded below by some constant $\delta$, uniformly in $h>0$. Moreover, by \eqref{norm of AhPh}, $\|A_h^\#\|_{\L(X^2_{0,h})}\le C h^{-2}$. Therefore,  $\sigma(A_h^\#) \subset [\delta, Ch^{-2}]$. 

Using a duality argument, it is clear that it suffices to prove \eqref{ineq:sectoriality of Ah} for $q \in [2,\infty)$. 
We choose the range $q \in [2,\infty)$ in order to utilize the estimates \eqref{ineq:stability of Ph on X12} and \eqref{Convergence estimate for Ph-Rh on X1/2}. Let $\nu \in (0,\pi/2)$ be fixed but arbitrary.  We begin by establishing the following: there exists  
$\omega_0 >0 $ such that 
	\begin{equation}\label{important claim1}
		\|A_h^\# R(\lambda,A_h^\#) u_h\|_{X^q_{0,h}} \le C (1+|\lambda|)^{-1/2} \|u_h\|_{X^q_{1/2}},\quad |\lambda| \le \omega_0 h^{-2}, \, \lambda \in \C \setminus \Sigma_\nu, \, u_h \in X^q_{0,h}.
	\end{equation}
Let us first show that \eqref{important claim1} implies \eqref{ineq:sectoriality of Ah} for $|\lambda| \le \omega_0 h^{-2}$. Since $P_h A_q^\# = A_h^\# R_h^\#$, we can write 
	\begin{equation}\label{eq:R(lambda,Ah) identity with R_h 1}
		 R(\lambda,A_h^\#)u_h = P_h R(\lambda,A_q^\#) u_h + A_h^\#R(\lambda,A_h^\#) (P_h-R_h^\#) R(\lambda,A_q^\#)u_h
	\end{equation}
	and consequently, 
	\begin{align*}
		\|R(\lambda,A_h^\#)u_h\|_{X^q_{0,h}} & \le \| P_h R(\lambda,A_q^\#) u_h\|_{X^q_{0,h}} + \| A_h^\#R(\lambda,A_h^\#) (P_h-R_h^\#) R(\lambda,A_q^\#)u_h\|_{X^q_{0,h}} 
		\\
		& \stackrel{\mathrm{(i)}}{\le} C (1+|\lambda|)^{-1} \|u_h\|_{X^q_{0,h}}  + C (1+|\lambda|)^{-1/2} \| (P_h-R_h^\#) R(\lambda,A_q^\#)u_h\|_{X^q_{1/2}}
		\\
		& \stackrel{\mathrm{(ii)}}{\le}  C (1+|\lambda|)^{-1} \|u_h\|_{X^q_{0,h}}  + C (1+|\lambda|)^{-1/2} \| R(\lambda,A_q^\#)u_h\|_{X^q_{1/2}}
		\\
		& \stackrel{\mathrm{(iii)}}{\le}  C (1+|\lambda|)^{-1} \|u_h\|_{X^q_{0,h}},
	\end{align*}
    where (i) follows from \eqref{stability of Ph on X0}, \eqref{sectoriality of Aq} and \eqref{important claim1}, (ii) follows from  \eqref{ineq:stability of Rh on Xq1/2} and \eqref{ineq:stability of Ph on X12}, and (iii) follows from \eqref{sectoriality of Aq interpolated}. 
    
    We now prove \eqref{important claim1}. 
    It suffices to show that 
    $$\|A_h^\# R(\lambda,A_h^\#)\|_{\L(X^q_{1/2}\cap S_h,X^q_{0,h})} \coloneq \sup_{u_h \in S_h \setminus \{0\}} \frac{\|A_h^\# R(\lambda,A_h^\#)u_h\|_{X^q_{0,h}}}{\|u_h\|_{X^q_{1/2}}} \le C (1+|\lambda|)^{-1/2}.$$    
    We do so via an absorption argument.
    By the resolvent identity $\lambda R(\lambda,S) = I + SR(\lambda, S)$ and \eqref{eq:R(lambda,Ah) identity with R_h 1}, we write 
	\begin{equation}\label{eq:R(lambda,Ah) second identity with R_h 1}
		 A_h^\#R(\lambda,A_h^\#)u_h = P_h A_q^\# R(\lambda,A_q^\#) u_h + \lambda A_h^\# R(\lambda,A_h^\#) (P_h-R_h^\#) R(\lambda,A_q^\#)u_h.
	\end{equation}
	By the stability of $P_h$ and \eqref{sectoriality of Aq interpolated}, we estimate 
	\begin{equation}\label{ineq:max norm claim proof estimate 1 1}
		 \| P_h A_q^\#R(\lambda,A_q^\#) u_h\|_{X^q_{0,h}} \le C (1+|\lambda|)^{-1/2} \|u_h\|_{X^q_{1/2}} , \quad \lambda \in \C \setminus \Sigma_\nu.
	\end{equation}
    To estimate the second term of \eqref{eq:R(lambda,Ah) second identity with R_h 1} we recall some standard inequalities.     
     By \eqref{Convergence estimate for Ph-Rh} and the inverse estimate \eqref{ineq:inverse estimate for Xq spaces} we get  
   $$
       \|P_h-R_h^\#\|_{\L(X^q_1,X^q_{1/2})} \le C h^{-1}     \|P_h-R_h^\#\|_{\L(X^q_1,X^q_0)} \le C h .
    $$ Moreover, by \eqref{ineq:stability of Ph on X12} and \eqref{ineq:stability of Rh on Xq1/2} we get  
    $\|P_h-R_h^\#\|_{\L(X^q_{1/2})}\le C$. Interpolating between these two inequalities we obtain 
    \begin{equation}\label{interpolated Ph-Rh}
        \|P_h-R_h^\#\|_{\L(X^q_{3/4},X^q_{1/2})}\le C h^{1/2}.
    \end{equation}
    We now estimate the second term of \eqref{eq:R(lambda,Ah) second identity with R_h 1} as follows 
	\begin{align*}
		&\| \lambda A_h^\# R(\lambda,A_h^\#) (P_h-R_h^\#) R(\lambda,A_q^\#)u_h \|_{X^q_{0,h}} 
        \\
        & \le |\lambda| \|A_h^\# R(\lambda,A_h^\#)\|_{\L(X^q_{1/2}\cap S_h,X^q_{0,h})} \|(P_h-R_h^\#)R(\lambda,A_q^\#) u_h\|_{X^q_{1/2}} 
		\\
		&  \stackrel{\mathrm{(i)}}{\le} C |\lambda| \|A_h^\# R(\lambda,A_h^\#)\|_{\L(X^q_{1/2}\cap S_h,X^q_{0,h})}  h^{1/2} \|R(\lambda,A_q^\#)u_h\|_{X^q_{3/4}}
		\\
		&  \stackrel{\mathrm{(ii)}}{\le}C |\lambda| \|A_h^\# R(\lambda,A_h^\#)\|_{\L(X^q_{1/2}\cap S_h,X^q_{0,h})} h^{1/2} (1+|\lambda|)^{-3/4} \|u_h\|_{X^q_{1/2}}
		\\
		& \le  C   \|A_h^\# R(\lambda,A_h^\#)\|_{\L(X^q_{1/2}\cap S_h,X^q_{0,h})} h^{1/2} (1+|\lambda|)^{1/4} \|u_h\|_{X^q_{1/2}},
	\end{align*}
    where in (i) we used \eqref{interpolated Ph-Rh} and in (ii) we used \eqref{sectoriality of Aq interpolated}. 
	Let $\omega_0>0$ such that $C h^{1/2} (1+|\lambda|)^{1/4} \le 1/2  $ when $|\lambda| \le \omega_0 h^{-2}$. Then the calculations above imply 
	\begin{equation}\label{ineq:max norm claim proof estimate 2 1}
			\| \lambda A_h^\# R(\lambda,A_h^\#) (P_h-R_h^\#) R(\lambda,A_q^\#)u_h \|_{X^q_{0,h}} \le \tfrac 12   \|A_h^\# R(\lambda,A_h^\#)\|_{\L(X^q_{1/2}\cap S_h,X^q_{0,h})} \|u_h\|_{X^q_{1/2}}.
	\end{equation}
    By \eqref{eq:R(lambda,Ah) second identity with R_h 1}, \eqref{ineq:max norm claim proof estimate 1 1} and \eqref{ineq:max norm claim proof estimate 2 1},
	we get that 
	$$\|A_h^\# R(\lambda,A_h^\#)\|_{\L(X^q_{1/2}\cap S_h,X^q_{0,h})} \le C(1+|\lambda|)^{-1/2} + \tfrac 12 \|A_h^\# R(\lambda,A_h^\#)\|_{\L(X^q_{1/2}\cap S_h,X^q_{0,h})}  ,$$
	proving \eqref{important claim1}.

    We now turn to the case $|\lambda|\ge \omega_0 h^{-2}$. 
    For $x \in \O_h$ let $G_h^x (\cdot,\bar \lambda)$ denote the adjoint discrete Green's function given by
	$$G_h^x (y,\bar \lambda) = (R(\bar \lambda,A_h^\#) \delta_h^x)(y), \quad \lambda \notin \Sigma_\nu,$$
	where $\delta_h^x$ is the discrete delta-function given by 
	$$(u_h,\delta_h^x )_{L^2} = u_h(x), \quad u_h\in S_h.$$ 
	We then have 
	$$(R(\lambda,A_h^\#) u_h)(x) = (u_h, G_h^x(\cdot,\bar \lambda))_{L^2}, \quad u_h \in S_h.$$
	Therefore, by Schur's lemma (cf.\,\cite[Appendix A.1]{GrafModern})  and noting that $G_h$ is symmetric in $x,y$, it suffices to show that there is $C>0$ such that for all $x \in \O_h$,
	\begin{equation}\label{ineq:L1 estimate of Ghx1}
		\|G_h^x(\cdot, \lambda)\|_{L^1} \le C |\lambda|^{-1}, \quad  |\lambda| \ge \omega_0 h^{-2}, \lambda \notin \Sigma_\nu.
	\end{equation}
    The proof of \eqref{ineq:L1 estimate of Ghx1} is similar to \cite[Section 3]{Bakaev-Thomee-maximum-norm} using weighted $L^2$ estimates and is thus omitted. 
\end{proof}

\begin{remark}\label{remark:end point sectoriality}
One can prove the endpoint cases $q \in \{1,\infty\}$ under the extra regularity assumption of $a_{ij}\in  C^{1,1/2}$ and $ \O\in  C^{2,1/2}$. Indeed, the case $q=\infty$ is proved as in  \cite{Bakaev-Thomee-maximum-norm} using Schauder estimates for $A^\#$ and the case $q=1$ follows by duality.        
\end{remark}

\subsubsection{Proof of Theorem \ref{theorem:Hcalc for Ah} and its corollaries} \label{sec: proof of calculus for principal part}
We now set the stage for the proof of Theorem \ref{theorem:Hcalc for Ah}. We follow the perturbation argument of Li and Zhou \cite{LiZhou2026}. By appealing directly to standard results from \cite{Haase:2}, which are re-derived in \cite[Section 2.1, Step 1]{LiZhou2026}, we are able to present a more concise proof. 

Since $A_h^\# P_h$ is bounded on $X^q_0$ with norm $\|A_h^\# P_h\|_{\L (X^q_0)} \le C h^{-2}$ (see \eqref{norm of AhPh}), while $A_q^\#$ is not, it is preferable to employ a perturbation argument using an operator that exhibits the same scaling as $A_h^\#$. That operator is the sectorial approximation $J_{h} \in \L(X^q_0)$ of $A_q^\#$ given by 
$$ J_h\coloneq (h^2+A_q^\#)(1+h^2 A_q^\#)^{-1} .$$
To see why $J_h$ exhibits the same scaling as $A_h^\# P_h$, using the sectoriality of $A_q^\#$, we estimate 
\begin{equation}\label{norm of Aqh}
\begin{aligned}
		\|J_h\|_{\L(X^q_0)}& \le h^2 \|(1+h^2A_q^\#)^{-1}\|_{\L(X^q_0)} + \|A_q^\#(1+h^2A_q^\#)^{-1}\|_{\L(X^q_0)}
		\\
		& = \|(h^{-2}+A_q^\#)^{-1}\|_{\L(X^q_0)} + h^{-2} \|A_q^\#(h^{-2}+A_q^\#)^{-1}\|_{\L(X^q_0)} \le C h^{-2}.
\end{aligned}
\end{equation}
Recall that $A_q^\#$ admits a bounded $H^\infty$-calculus on $X^q_0$ of angle zero (see Section \ref{section:The operator Aq}). By \cite[Proposition 2.1.3]{Haase:2}, $J_h$ is a sectorial operator on $X^q_0$ with angle zero and by \cite[Proposition 5.3.4 and Lemma 5.4.5]{Haase:2}, $J_h$ has a bounded $H^\infty$-calculus of angle zero. Moreover, the $H^\infty$-calculus constant of $J_h$ is equal to that of $A_q^\#$ and is thus independent of $h>0$. Hence, there is $C>0$ independent of $h$ such that 
\begin{equation}\label{ineq:calculus for Jh}
    \|f(J_h)\|_{\L(X^q_{0})} \le C \|f\|_{H^\infty(\Sigma_\nu)}, \quad f \in H^\infty(\Sigma_\nu) \cap H^1(\Sigma_\nu) .
\end{equation}
Moreover, by elementary calculations we have
\begin{equation}\label{eq:formula for R(z,Aqh)}
		R(z,J_h) = \frac{1}{1-h^2z}(1+h^2A_q^\#)  \left( \frac{z-h^2}{1-h^2z} -A_q^\# \right)^{-1}, \quad z \in \C \setminus \Sigma_{\nu}.
\end{equation}
In order to prove Theorem \ref{theorem:Hcalc for Ah} we need the following resolvent estimates:
\begin{lemma}
    For every $\nu \in (0,\pi/2)$ there is a constant $C>0$, independent of $h$, such that 
    \begin{align}
		\|R(z,A_q^\#)-R(z,J_h)\|_{\L(X^q_0)} &\le C h^{2}, \quad z \in \partial \Sigma_\nu  \label{ineq:Resolvent estimate Aq-Aqh}
        \\
        \|R(z,A_q^\#)-R(z,A_h^\#)P_h\|_{\L(X^q_0)}&\le C h^2, \quad z \in \C \setminus \Sigma_\nu.        \label{ineq:Resolvent estimate Aq-Ah}
	\end{align}
\end{lemma}
\begin{proof} We first prove \eqref{ineq:Resolvent estimate Aq-Aqh}. 
By the identity $A_q^\#-J_h = h^2((A_q^\#)^2-1)(1+h^2A_q^\#)^{-1}$ and \eqref{eq:formula for R(z,Aqh)}, we get 
\begin{align*}
	R(z,A_q^\#) - R(z,J_h) &= R(z,J_h) (A_q^\#-J_h) R(z,A_q^\#)  
	\\
	&= \frac{h^2}{1-h^2z}((A_q^\#)^2-1)( \frac{z-h^2}{1-h^2z}-A_q^\#)^{-1} R(z,A_q^\#).
\end{align*}
Let $z \in \partial \Sigma_\nu$ and let $\xi = \frac{z-h^2}{1-h^2z}$. Note that $\Re \xi = |1-h^2z|^{-2}((1+h^4)\Re z -h^2(1+|z|^2))$ and $\Im \xi = |1-h^2z|^{-2}(1-h^4) \Im z $. Hence,  $\xi \in \C \setminus \Sigma_\nu$.  Therefore, the sectoriality of $A_q^\#$ implies 
$$ \|R(z,A_q^\#) - R(z,J_h) \|_{\L(X^q_0)} \le C \frac{h^2}{|1-h^2z|}  \le C h^2, $$
where in the last estimate we used the elementary inequality
$$\frac 1{|1-h^2z|^2} = \frac{1}{1-2h^2|z|\cos\nu +h^4|z|^2} \le \frac 1{(1-\cos\nu)(1+h^4|z|^2)}\le C.$$
We now prove \eqref{ineq:Resolvent estimate Aq-Ah}. Let $z \in \C \setminus \Sigma_\nu$. By the identity 
\begin{equation*} 
    R(z,A_q^\#)-R(z,A_h^\#)P_h= (I-P_h)R(z,A_q^\#) + P_h R(z, A_q^\#)-R(z, A_h^\#) P_h
\end{equation*}
and the triangle inequality, it suffices to establish that
$$  \|(I-P_h)R(z,A_q^\#)\|_{\L(X^q_0)} + \| R(z, A_h^\#) P_h - P_h R(z, A_q^\#) \|_{\L(X^q_0)} \le C h^2.$$
Note that by the convergence estimate \eqref{Convergence estimate for I-Ph} and \eqref{sectoriality of Aq} we get
   \begin{align*}
       \|(I-P_h)R(z,A_q^\#)\|_{\L(X^q_0)} &\le \|I-P_h\|_{\L(X^q_1,X^q_0)} \|R(z,A_q^\#)\|_{\L(X^q_0,X^q_1)} \le Ch^2  .
   \end{align*}
  For the last term we calculate
    \begin{align*}
R(z, A_h^\#) P_h - P_h R(z, A_q^\#) 
&= R(z, A_h^\#) ( P_h (z - A_q^\#) - (z - A_h^\#) P_h ) R(z, A_q^\#) \\
&= R(z, A_h^\#) ( A_h^\# P_h - P_h A_q^\# ) R(z, A_q^\#) \\
&= R(z, A_h^\#) A_h^\# ( P_h (A_q^\#)^{-1} - (A_h^\#)^{-1} P_h ) A_q^\# R(z, A_q^\#)
\end{align*}
and thus estimate
\begin{align*}
& \| R(z, A_h^\#) P_h - P_h R(z, A_q^\#) \|_{\L(X^q_0)} 
\\
&\leq \| R(z, A_h^\#) A_h^\#\|_{\L(X^q_{0,h})} 
\| P_h (A_q^\#)^{-1} - (A_h^\#)^{-1} P_h \|_{\L(X^q_0)} 
\| A_q^\# R(z, A_q^\#) \|_{\L(X^q_0)} 
\\
& \stackrel{\mathrm{(i)}}{\le}  C  \| P_h (A_q^\#)^{-1} - (A_h^\#)^{-1} P_h \|_{\L(X^q_0)}  \stackrel{\mathrm{(ii)}}{=} C\| (P_h - R_h^\#) (A_q^\#)^{-1} \|_{\L(X^q_0)}
\\
&= C \|P_h-R_h^\#\|_{\L(X^q_1,X^q_0)}  \stackrel{\mathrm{(iii)}}{\le}  C h^2,
\end{align*}
where (i) follows from the sectoriality of $A_q^\#$ and $A_h^\#$ (see Lemma \ref{lemma:sectoriality of Ah}), (ii) from the identity $A_h^\# R_h^\#=P_h A_q^\#$, and  (iii) from \eqref{Convergence estimate for Ph-Rh}.
\end{proof}

We are now ready to present the proof of Theorem \ref{theorem:Hcalc for Ah}. 
\begin{proof}[Proof of Theorem \ref{theorem:Hcalc for Ah}]
By virtue of \eqref{ineq:calculus for Jh}, to establish \eqref{proof of main thm sufficient condition 1} it suffices to find a constant $C>0$, independent of $h$, such that
\begin{equation}\label{proof of main thm sufficient condition 2}
     \|f(A_h^\#)P_h - f(J_h)\|_{\L(X^q_0)} \le C \|f\|_{H^\infty(\Sigma_\nu)},  \quad f \in H^\infty(\Sigma_\nu) \cap H^1(\Sigma_\nu) ,  
\end{equation}
Recall that 
	\begin{equation} \label{eq:calculus integral def}
		 f(A_h^\#)P_h - f(J_h) = \frac{1}{2\pi i} \int_{\partial \Sigma_\sigma} f(z) (R(z,A_h^\#)P_h-R(z,J_h)) \,dz,
	\end{equation}
	where $0<\sigma <\nu$.
     By \eqref{norm of AhPh} and \eqref{norm of Aqh}, there is $c^*>0$ such that 
	\begin{equation}\label{ineq:bound for both norms}
		c^* h^{-2} \ge 2 \max \{ \|A_h^\#\|_{\L(X^q_{0,h})} , \|J_h\|_{\L(X^q_0)}\},
	\end{equation}
and thus $$\rho(J_h) \cap \rho(A_h^\#) \supset \{z \in \C \colon \Re z \ge c^*h^{-2}\}.$$
Moreover, as 
$$ \rho(J_h) \cap \rho(A_h^\#) \supset \C \setminus (0,\infty),$$
we obtain that 
\begin{equation}
	 \rho(J_h) \cap \rho(A_h^\#) \supset \C \setminus (0,c^*h^{-2}).
\end{equation}
Hence, by Cauchy's theorem we can deform the contour of integration in \eqref{eq:calculus integral def} and write 
	\begin{equation} \label{eq:deformed contour}
	f(A_h^\#)P_h - f(J_h)  \eqcolon I_1 +I_2 + I_3,
\end{equation}
where 
\begin{equation*}
	I_k\coloneq  \frac{1}{2\pi i} \int_{\Gamma_k}  f(z) (R(z,A_h^\#)P_h-R(z,J_h)) \,dz, \quad k=1,2,3,
\end{equation*}
and the contours $\Gamma_k$ are given by 
\begin{align*}
\Gamma_1 &: \text{the line segment from } \frac{c^*}{h^2 \cos\nu} e^{i\nu} \text{ to } 0,
\\
\Gamma_2 &: \text{the line segment from } 0 \text{ to } \frac{c^*}{h^2 \cos\nu} e^{-i\nu},
\\
\Gamma_3 &: \text{the vertical segment from } \frac{c^*}{h^2} - i \frac{c^*}{h^2} \tan\nu \text{ to } \frac{c^*}{h^2} + i \frac{c^*}{h^2} \tan\nu,
\end{align*}
see Figure \ref{fig:contours}.

By the resolvent estimates  \eqref{ineq:Resolvent estimate Aq-Aqh} and \eqref{ineq:Resolvent estimate Aq-Ah}, we obtain 
\begin{align*}
    \|I_1+I_2\|_{\L(X^q_0)} &\le \frac{1}{2\pi} \sum_{k=1}^2 \int_{\Gamma_k} |f(z)| \|R(z,A_h^\#)P_h-R(z,J_h)\|_{\L(X^q_0)} \, |dz| 
    \\
    &\le C h^2 \frac{c^*}{h^2 \cos \nu}  \|f\|_{H^\infty(\Sigma_\nu)}  \le C \|f\|_{H^\infty(\Sigma_\nu)} .  
\end{align*} 
Finally, for $z \in \Gamma_3$ we have that $|z| \ge c^* h^{-2} $ and thus \eqref{ineq:bound for both norms} and \eqref{ineq:stability of Ph from Xq0 to Xq0h} imply  
\begin{align*}
	\big\|R(z,A_h^\#) P_h - R(z, J_h)\big\|_{\L(X^q_0)}
	&\le 
	\|R(z,A_h^\#) P_h \|_{\L(X^q_0)}
	+ \|R(z, J_h)\|_{\L(X^q_0)}  
	\\
	&\le 
	C\|R(z,A_h^\#) \|_{\L(X^q_{0,h})}
	+ \|R(z, J_h)\|_{\L(X^q_0)} 
	\\
	&\le 
	\tfrac{C}{|z| - \|A_h^\#\|_{\L(X^q_{0,h})}} 
	+ \tfrac{1}{|z| - \|J_h\|_{\L(X^q_0)}}  \le C h^{2} ,
\end{align*} 
and thus $\|I_3\|_{\L(X^q_0)} \le C\|f\|_{H^\infty(\Sigma_\nu)}.$

\begin{figure}[h!]
	\centering
	\begin{tikzpicture}[scale=2,
		every path/.style={>=stealth},
		gamma/.style={thick, postaction={decorate},
			decoration={markings, mark=at position 0.6 with {\arrow{>}}}},
		gammadashed/.style={thick, dashed, postaction={decorate},
			decoration={markings, mark=at position 0.6 with {\arrow{>}}}}
		]
		
		\def\thet{30}       
		\def\Rzero{1.5}     
		
		\coordinate (A) at ({\Rzero*cos(\thet)},{\Rzero*sin(\thet)});
		\coordinate (B) at ({\Rzero*cos(\thet)},{-\Rzero*sin(\thet)});
		
		\draw[->] (-0.2,0) -- (2.2,0) node[right] {$\mathrm{Re}\, z$};
		\draw[->] (0,-1.5) -- (0,1.5) node[left] {$\mathrm{Im}\, z$};
		
		\draw[gamma] (A) -- (0,0) node[midway, above left] {$\Gamma_1$};
		
		\draw[gammadashed] (B) -- (A) node[midway, right] {$\Gamma_3$};
		
		\draw[gammadashed] (0,0) -- (B) node[midway, below left] {$\Gamma_2$};
		
		\draw (0.7,0) arc (0:\thet:0.7);
		\node at (0.55,0.15) {$\nu$};
		
		\node at (A) [above right] {$\displaystyle \frac{c^*}{h^2\cos\nu}\, e^{i\nu}$};
		\node at (B) [below right] {$\displaystyle \frac{c^*}{h^2\cos\nu}\, e^{-i\nu}$};
		
	\end{tikzpicture}

	\caption{The contours $\Gamma_1$, $\Gamma_2$, and $\Gamma_3$ in the sector $\Sigma_\nu$.}
	\label{fig:contours}
\end{figure}
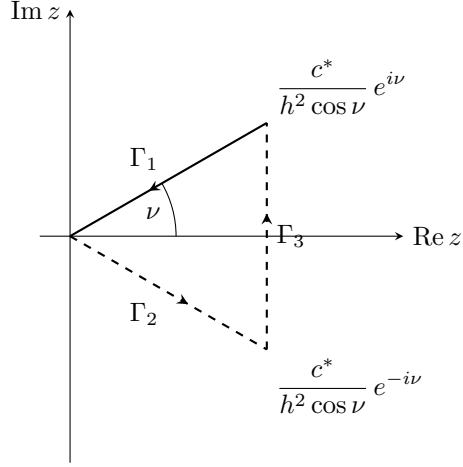
\end{proof}

\begin{remark} \label{remark:necessity of Rh}
     We do not know whether the stability of the Ritz projection $R_h$ is necessary for the proof of Theorem \ref{theorem:Hcalc for Ah}. However, it is necessary that $R_h$ is stable to have the norm characterization \eqref{Xah norm equiv Xa norm} of the abstract discrete fractional spaces given in Corollary \ref{cor: Xah norm equiv Xa norm}. Indeed, suppose that \eqref{Xah norm equiv Xa norm}  holds. In particular, 
\begin{equation}\label{ineq: discrete Riesz}
    \|(A_q^\#)^{1/2} (A_h^\#)^{-1/2} u_h\|_{X^q_0} \le C \|u_h\|_{X^q_{0,h}}, \quad u_h \in S_h.
\end{equation}
By \eqref{Xah norm equiv Xa norm}, in order to show that $R_h$ is stable on $X^q_{1/2}$, it suffices to prove
    $$ \|(A_h^\#)^{-1/2} P_h (A_q^\#)^{1/2} u \|_{X^q_{0,h}} \le C \|u\|_{X^q_0} ,\quad u \in X^q_0.$$
    We prove the latter by a duality argument. For $u \in X^q_0$ and $v \in X^{q'}_0$, using the definition and stability of $P_h$ and \eqref{ineq: discrete Riesz} we estimate 
    \begin{align*}
        ((A_h^\#)^{-1/2}P_h (A_q^\#)^{1/2} u,v)_{L^2} & = ((A_h^\#)^{-1/2}P_h (A_q^\#)^{1/2} u,P_hv)_{L^2} =(P_h (A_q^\#)^{1/2} u,(A_h^\#)^{-1/2} P_h v)_{L^2} 
        \\
        &= ((A_q^\#)^{1/2} u,(A_h^\#)^{-1/2} P_h v)_{L^2}
        =(u,(A_{q'}^\#)^{1/2}  (A_h^\#)^{-1/2} P_h v)_{L^2} 
        \\
        &\le \|u\|_{X^q_0} \|(A_{q'}^\#)^{1/2} (A_h^\#)^{-1/2} P_h v\|_{X^{q'}_0} \le C \|u\|_{X^q_0} \|P_h v\|_{X^{q'}_0} \\
        &\le C\|u\|_{X^q_0} \|v\|_{X^{q'}_0}. 
    \end{align*}
\end{remark}
We now establish the corollaries of Theorem \ref{theorem:Hcalc for Ah} stated in Section \ref{sec:principal part of Ah}. We begin with the proof of Corollary \ref{cor: Xah norm equiv Xa norm}, which provides a characterization of the discrete fractional spaces.
\begin{proof}[Proof of Corollary \ref{cor: Xah norm equiv Xa norm}]
The stability \eqref{Ph stable Xqbeta to Xqbetah} holds for $\alpha \in \{0,1\}$ by 
 \eqref{ineq:stability of Ph from Xq0 to Xq0h} and \eqref{Ph stable X1 to X1h}. Therefore,  Theorem \ref{theorem:Hcalc for Ah}, Lemma \ref{lem:BIP} and complex interpolation give the range $\alpha \in (0,1)$.
Evidently, \eqref{Xah norm equiv Xa norm} holds for $\alpha=0$ with $C=1$. 
    By complex interpolation, it suffices to prove \eqref{Xah norm equiv Xa norm} for $\alpha=1/2$. Moreover, by \eqref{Ph stable Xqbeta to Xqbetah} we get 
    $\|u_h\|_{X^{q,\#}_{1/2,h}} \le C \|u_h\|_{X^{q,\#}_{1/2}}$ for every $u_h \in S_h$. Hence, it remains to prove
    	$$ \|u_h\|_{X^{q,\#}_{1/2}} \le C \|u_h\|_{X^{q,\#}_{1/2,h}},\quad u_h \in S_h.$$
        This follows from a duality argument similar to \cite[Lemma 4.6 (iv)]{LiZhou2025_fully_discrete_AC}, combining the stability of $P_h$ from $X^{q',\#}_{1/2}$ to $X^{q',\#}_{1/2,h}$ (see \eqref{Ph stable Xqbeta to Xqbetah}) with the stability of the Ritz projection $R_h^\# = (A_h^\#)^{-1} P_h A_q^\#$ on $X^{q',\#}_{1/2}$ provided by \eqref{ineq:stability of Rh on Xq1/2}.
\end{proof}

We now prove Corollary \ref{cor:discrete vs continuous real int} which identifies certain real interpolation scales of the discrete fractional spaces. 

\begin{proof}[Proof of Corollary \ref{cor:discrete vs continuous real int}] To prove \eqref{ineq: equivalence of real interpolation for positive beta} we note that by \eqref{Xah norm equiv Xa norm} the inclusion map $\iota_h$ is stable from $X^{q,\#}_{\beta,h}$ to $X^{q,\#}_\beta$ and from $X^{q,\#}_{\alpha ,h}$ to $X^{q,\#}_\alpha $. Hence, by real interpolation, there is $C>0$ independent of $h$ such that 
$$ \|\iota_h\|_{\L \big( (X^{q,\#}_{\beta,h},X^{q,\#}_{\alpha ,h})_{\theta,p}, (X^{q,\#}_\beta,X^{q,\#}_\alpha )_{\theta,p} \big )} \le C,$$
and consequently, 
\begin{align*}
\|u_h\|_{(X^{q,\#}_\beta,X^{q,\#}_\alpha )_{\theta,p}} =\|\iota _hu_h\|_{(X^{q,\#}_\beta,X^{q,\#}_\alpha )_{\theta,p}}  &\le \|\iota_h\|_{\L \big( (X^{q,\#}_{\beta,h},X^{q,\#}_{\alpha ,h})_{\theta,p}, (X^{q,\#}_\beta,X^{q,\#}_\alpha )_{\theta,p} \big )}\|u_h\|_{(X^{q,\#}_{\beta,h},X^{q,\#}_{\alpha ,h})_{\theta,p}} 
\\
&\le C \|u_h\|_{(X^{q,\#}_{\beta,h},X^{q,\#}_{\alpha ,h})_{\theta,p}} .
\end{align*}
Moreover, by \eqref{Ph stable Xqbeta to Xqbetah} and real interpolation there is $C>0$, independent of $h$, such that 
$$ \|P_h\|_{\L \big( (X^{q,\#}_\beta,X^{q,\#}_\alpha )_{\theta,p}, (X^{q,\#}_{\beta,h},X^{q,\#}_{\alpha ,h})_{\theta,p} \big )}  \le C$$
and thus 
\begin{align*}
	\|u_h\|_{(X^{q,\#}_{\beta,h},X^{q,\#}_{\alpha ,h})_{\theta,p}} &=\|P_h u_h\|_{(X^{q,\#}_{\beta,h},X^{q,\#}_{\alpha ,h})_{\theta,p}}  \le \|P_h\|_{\L \big( (X^{q,\#}_{\beta},X^{q,\#}_{\alpha })_{\theta,p}, (X^{q,\#}_{\beta,h},X^{q,\#}_{\alpha ,h})_{\theta,p} \big )} \|u_h\|_{(X^{q,\#}_\beta,X^{q,\#}_\alpha )_{\theta,p}} 
	\\
	&\le C  \|u_h\|_{(X^{q,\#}_\beta,X^{q,\#}_\alpha )_{\theta,p}}.
\end{align*}
This proves \eqref{ineq: equivalence of real interpolation for positive beta}. 

We now prove \eqref{ineq: equivalence of real interpolation for trace space}. Note that \eqref{ineq: equivalence of real interpolation for trace space} trivially holds for $\theta=0$ with $C=1$. By Lemma \ref{lemma:sectoriality of Ah} and standard semigroup theory, there is $C>0$ independent of $h$ such that $\sup_{t>0} \| \exp{-tA_h^\#} \|_{\L(X^{q,\#}_{0,h})}+\|t A_h^\# \exp{-tA_h^\#} \|_{\L(X^{q,\#}_{0,h})} \le C$ and thus by \cite[Proposition 2.2.15]{Lunardi_analytic_book} the embedding
$$ (X^{q,\#}_{0,h} , X^{q,\#}_{1,h})_{1/2,1} \hookrightarrow X^{q,\#}_{1/2,h} \hookrightarrow (X^{q,\#}_{0,h} , X^{q,\#}_{1,h})_{1/2,\infty}$$
holds with constants that do not depend on $h$. Therefore, by the reiteration theorem (cf.\,\cite[Theorem L.3.1]{analysis_volume_3}) we get 
$$ (X^{q,\#}_{0,h} ,X^{q,\#}_{1,h})_{\theta,p} = (X^{q,\#}_{0,h} ,X^{q,\#}_{1/2,h})_{2\theta,p} $$
with equivalent norms and constants that are independent of $h$. Since the reiteration theorem also gives 
$$(X^{q,\#}_{0} ,X^{q,\#}_{1})_{\theta,p} =  (X^{q,\#}_{0} ,X^{q,\#}_{1/2})_{2\theta,p} $$
with equivalent norms, \eqref{ineq: equivalence of real interpolation for positive beta} implies the desired result.  

\end{proof}

\subsection{$H^\infty$-calculus for the full operator $A_h$}\label{sec:with lower-order}
In this section, we extend the results of Section \ref{sec:principal part of Ah} to operators with lower-order terms and give the proof of Theorem \ref{thm:H-calc for Ah}. 
\begin{proof}[Proof of Theorem \ref{thm:H-calc for Ah}]
 By a standard perturbation argument (cf.\,\cite[Theorem 16.2.7]{analysis_volume_3}), it suffices to show that there is $C>0$, independent of $h$, such that
\begin{equation}\label{ineq: Bh bounded by sqrt of Ah}
    \|B_hu_h\|_{X^q_{0,h}} \le C \|(A_h^\#)^{1/2} u_h\|_{X^q_{0,h}} , \quad u_h \in S_h.
\end{equation}
By the norm equivalence \eqref{Xah norm equiv Xa norm}, and noting that $X^q_{1/2}=D((A^\#_q)^{1/2})=W^{1,q}_0$ with equivalent norms by \eqref{eq:fraction space equals both domains} and \eqref{X12 characterization}, we have 
$$
\|(A_h^\#)^{1/2}u_h\|_{X^q_{0,h}} \eqsim \|(A^\#_q)^{1/2}u_h\|_{X^q_0} \eqsim   \|u_h\|_{X^q_{1/2}}
$$
with constants independent of $h$. Hence, 
\begin{equation*}
    \|B_hu_h\|_{X^q_{0,h}} \le C \| u_h\|_{W^{1,q}} \le C \|u_h\|_{X^q_{1/2}}\le C  \|(A_h^\#)^{1/2}u_h\|_{X^q_{0,h}}, \quad u_h \in S_h.
\end{equation*}
This  proves \eqref{ineq: Bh bounded by sqrt of Ah}.
\end{proof}
From now on we fix $\nu \in (0,\pi/2)$ and let $\lambda_0 \ge0$ be as in Theorem \ref{thm:H-calc for Ah}.
For $\alpha \in (0,1]$ we define 
$$X^q_{\alpha,h} \text{ to be  $S_h$ with the $\|(\lambda_0+A_h)^\alpha\cdot\|_{X^q_0}$-norm.} $$
We now prove a discrete version of \eqref{eq:domains equal with lower-order terms}.
\begin{lemma}\label{lemma:norm equivalence for lower-order terms} Let $q \in (1,\infty)$ and $\alpha \in (0,1]$. Suppose that Assumption \ref{main assumptions} holds for $A^\#$. Then, with $\lambda_0 \ge0$ as in Theorem \ref{thm:H-calc for Ah}, there is a constant $C>0$, independent of $h$, such that 
\begin{equation}\label{ineq: norm equivalence with lower-order terms}
  C^{-1} \|(A_h^\#)^\alpha  u_h \|_{X^q_{0,h}} \le     \| (\lambda_0 + A_h)^\alpha u_h\|_{X^q_{0,h}} \le C \|(A_h^\#)^\alpha u_h \|_{X^q_{0,h}} ,\quad u_h \in S_h.
\end{equation}
\end{lemma}
\begin{proof}
By Theorems \ref{thm:H-calc for Ah} and \ref{theorem:Hcalc for Ah}, both $A_h^\#$ and $ \lambda_0+A_h$ admit a bounded $H^\infty(\Sigma_\nu)$-calculus with a constant independent of $h$. Hence, using complex interpolation, it suffices to prove \eqref{ineq: norm equivalence with lower-order terms} for $\alpha=1$. Moreover, since $\|(A^\#_h)^{-1}\|_{\L(X^q_{0,h})}+ \| (\lambda_0+A_h)^{-1}\|_{\L(X^q_{0,h})} \le C$ is uniformly bounded in $h$, it suffices to show that the equivalence of the inhomogeneous norms
\begin{equation}\label{ineq: norm equivalence with lower-order terms nonhomogeneous}
 \|A_h^\#  u_h \|_{X^q_{0,h}} + \|u_h\|_{X^q_{0,h}} \eqsim    \| (\lambda_0 + A_h) u_h\|_{X^q_{0,h}} + \|u_h\|_{X^q_{0,h}} ,\quad u_h \in S_h
\end{equation}
holds with constants independent of $h$.  By complex interpolation and Young's inequality, we have
$$ \|(A_h^\#)^{1/2} u_h\|_{X^q_{0,h}} \le C \|A_h^\# u_h\|_{X^q_{0,h}}^{1/2}\| u_h\|_{X^q_{0,h}}^{1/2} \le  C\eps \| A_h^\# u_h\|_{X^q_{0,h}} + C_\eps \| u_h\|_{X^q_{0,h}},  $$
which, together with \eqref{ineq: Bh bounded by sqrt of Ah}, implies  
\begin{equation}\label{ineq: Bh kickback}
    \|B_h u_h\|_{X^q_{0,h}} \le C\eps \| A_h^\# u_h\|_{X^q_{0,h}} + C_\eps \| u_h\|_{X^q_{0,h}}.
\end{equation}
A standard absorption argument using \eqref{ineq: Bh kickback} gives \eqref{ineq: norm equivalence with lower-order terms nonhomogeneous}.

\end{proof}

Lemma \ref{lemma:norm equivalence for lower-order terms} implies that the discrete fractional spaces $ X^q_{\alpha,h} $ are completely determined by the principal part of $A_h$. Moreover, it allows us to immediately obtain the following characterizations of the discrete fractional spaces.

\begin{corollary}\label{cor: Xah norm equiv Xa norm with lower-order terms}
For every $\alpha \in [0,1]$ there is a constant $C>0$, independent of $h$, such that 
\begin{equation}\label{Ph stable Xqbeta to Xqbetah with lower-order}
    \|P_h\|_{\L( X^q_\alpha ,  X^q_{\alpha,h})} \le C.
\end{equation}
   For every $\alpha \in [0,1/2]$ there is $C>0$ independent of $h$ such that 
   \begin{equation}\label{Xah norm equiv Xa norm with lower-order}
			C^{-1} \|u_h\|_{ X^q_{\alpha}} \le   \|u_h\|_{ X^q_{\alpha,h}} \le C\|u_h\|_{ X^q_{\alpha}} , \quad u_h \in S_h.
		\end{equation} 
\end{corollary}
\begin{proof}
    This follows from Corollary \ref{cor: Xah norm equiv Xa norm} and Lemma \ref{lemma:norm equivalence for lower-order terms}.
\end{proof}

\begin{corollary} \label{cor:discrete vs continuous real int with lower-order terms}
	Let $q \in (1,\infty)$, $p \in [1,\infty]$ and $\theta \in (0,1)$. Then there is $C>0$, independent of $h>0$, such that the following hold:
    \begin{enumerate}
        \item [$(\mathrm{i})$] For $0 \le \beta <\alpha\le 1/2$ and $u_h \in S_h$,
	\begin{equation}\label{ineq: equivalence of real interpolation for positive beta with lower-order}
	    C^{-1} \|u_h\|_{( X^q_{\beta,h}, X^q_{\alpha,h})_{\theta,p}} \le \|u_h\|_{( X^q_{\beta}, X^q_{\alpha})_{\theta,p}} \le C\|u_h\|_{( X^q_{\beta,h}, X^q_{\alpha,h})_{\theta,p}}. 
	\end{equation}
   
    \item[$(\mathrm{ii})$] If $\theta \in [0,1/2)$ then for every $u_h \in S_h$, 
    \begin{equation}\label{ineq: equivalence of real interpolation for trace space with lower-order}
	    C^{-1} \|u_h\|_{( X^q_{0,h}, X^q_{1,h})_{\theta,p}} \le \|u_h\|_{( X^q_{0}, X^q_{1})_{\theta,p}} \le C\|u_h\|_{( X^q_{0,h}, X^q_{1,h})_{\theta,p}}. 
	\end{equation}
    \end{enumerate}
\end{corollary}
\begin{proof}
    This follows from Corollary \ref{cor:discrete vs continuous real int} and Lemma \ref{lemma:norm equivalence for lower-order terms}.
\end{proof}

\section{Fully discrete  stochastic maximal regularity}\label{sec:fully DSMR}
We now prove the second main result of the paper, namely the fully discrete stochastic maximal regularity for fully discrete approximation schemes of 
\begin{equation}\label{eq:Linear shifted SPDE}
     \begin{cases}
 du (t) + (\lambda_0+A)u(t) \, dt =  g(t) \,dW(t), \quad t \in (0,T),
 \\
 u(0) =0,
 \end{cases}
\end{equation}
where $T \in (0,\infty]$ and $\lambda_0 \ge0$ is large enough so that $\lambda_0+A_h$ admits a bounded $H^\infty$-calculus, uniformly in the mesh size $h$; see Theorem \ref{thm:H-calc for Ah}.  
 Recall that $\O$ is a bounded, convex domain in $\R^3$ of class $C^2$,  $A = -\nabla \cdot a\nabla + b \cdot \nabla +c$ with Dirichlet boundary conditions and 
 we require that the principal part $A^\# \coloneq -\nabla \cdot a\nabla$ of $A$ satisfies Assumption \ref{main assumptions}. We note that if the lower-order coefficients $b_i=c=0$ then one can take $\lambda_0=0$.

 We consider the full discretization of \eqref{eq:Linear shifted SPDE} given by
\begin{equation} \label{Eq:Definition of approximation scheme fully discrete}
	\begin{cases}
		U^h_{n+1} \coloneq  r(\tau (\lambda_0+A_h)) U^h_n +  r(\tau (\lambda_0+A_h)) \int_{t_{n}}^{t_{n+1}} P_h g(s) \, dW(s), \quad n=0,\dots, N-1,
		\\
		U^h_0\coloneq 0,
	\end{cases}
	\end{equation}	
where $N=T/\tau$ if $T<\infty$ and $N=\infty$ if $T=\infty$, and $r$ is either the exponential function $r(z) \coloneq e^{-z}$, or an $A$-stable rational function that is consistent of order $\ell \ge 1$ and satisfies $r(\infty) = 0$, see Section \ref{section:SMR} for the definitions. As usual, we denote $L^q$ by $X^q_0$ and define $X^q_{0,h}$ to be $S_h$ with the $L^q$-norm. 
For $\alpha \in (0,1]$ we let $X^q_{\alpha} \coloneq  D((\lambda_0+A_q)^\alpha)$ endowed with the homogeneous graph norm and define  $X^q_{\alpha,h}$ to be $S_h$ with the $\|(\lambda_0+A_h)^{\alpha}\cdot\|_{L^q}$-norm.
Note that by \eqref{eq:fraction space equals both domains} and \eqref{X12 characterization}, $$X^q_1 = W^{2,q}\cap W^{1,q}_0 \text { and } X^q_{1/2}=W^{1,q}_0 \text{ with equivalent norms}.$$
Moreover, by Corollary \ref{cor: Xah norm equiv Xa norm with lower-order terms}, there is $C>0$ independent of $h$ such that 
$$     C^{-1} \|u_h\|_{X^q_{1/2,h}} \le \|u_h\|_{X^q_{1/2}} \le C \|u_h\|_{X^q_{1/2,h}}, \quad u_h \in S_h.
$$

The main result is formalized in the following theorem.
\begin{theorem}[Fully discrete SMR] \label{thm:fully discrete SMR}
Let $q \in [2,\infty)$ and $p \in (2,\infty)$, where we allow $p=2$ if $q=2$. Suppose that Assumption \ref{main assumptions} holds for the principal part $A^\#$ of $A$. Then there exists a constant $C>0$, independent of $h$ and $\tau$, such that for every $g \in L^p_{\mathbb F}(\Omega \times (0,T); \gamma(H, X^q_{0}))$, the fully discrete solution $(U^h_n)_{n=0}^{N}$ given in \eqref{Eq:Definition of approximation scheme fully discrete} satisfies
 \begin{equation}\label{ineq:space-time DSMR with discrete trace space}
      \begin{aligned}
          \Big (  \E \sup_{n=1,\dots, N} \| U^h_n \|^p_{(X^q_{0,h},X^q_{1,h})_{1/2-1/p,p}} \Big)^{1/p}& +  \Big( \E \sum_{n=1}^{N-1} \tau  \|  U^h_n \|^p_{X^q_{1/2,h}} \Big)^{1/p} 
          \\
          & \le C \|P_h g\|_{L^p(\Omega \times (0,T); \gamma(H, X^q_{0,h}))} 
      \end{aligned}
\end{equation}
and
   \begin{equation}\label{ineq:space-time DSMR}
         \begin{aligned}
      \Big( \E \sup_{n=1,\dots, N} \| U^h_n \|^p_{(X^q_0,X^q_1)_{1/2-1/p,p}} \Big )^{1/p}&+  \Big( \E \sum_{n=1}^{N-1} \tau  \|  U^h_n \|^p_{X^q_{1/2}} \Big)^{1/p} 
      \\
      &\le C \|P_h g\|_{L^p(\Omega \times (0,T); \gamma(H, X^q_{0,h}))} .
      \end{aligned}
   \end{equation}
   Furthermore, the right-hand sides of both estimates can be replaced by  $C\|g\|_{L^p(\Omega \times (0,T); \gamma(H, X^q_{0}))}$. 
\end{theorem}
\begin{proof}
By Theorem \ref{thm:H-calc for Ah}, $\lambda_0 + A_h$ admits a bounded $H^\infty$-calculus on $X^q_{0,h}$ with a constant independent of $h$. Hence, by Theorem \ref{thm:DSMR in Lq spaces}  there is a constant $C>0$, independent of $h$ and $\tau$, such that \eqref{ineq:space-time DSMR with discrete trace space} holds. The inequality \eqref{ineq:space-time DSMR} follows from  \eqref{ineq:space-time DSMR with discrete trace space} and the norm equivalences \eqref{ineq: equivalence of real interpolation for trace space with lower-order},  \eqref {Xah norm equiv Xa norm with lower-order}. The final remark follows from 
\begin{equation}\label{ineq:ideal property}
        \|P_hg\|_{L^p(\Omega \times (0,T); \gamma(H, X^q_{0,h}))}  \le  C  \|g\|_{L^p(\Omega \times (0,T); \gamma(H, X^q_{0}))}.
    \end{equation}
This inequality can be deduced from the stability of $P_h \colon X^q_0 \to X^q_{0,h}$ (see \eqref{ineq:stability of Ph from Xq0 to Xq0h}) and the ideal property (cf.\,\cite[Theorem 9.1.10]{analysis_volume_2}).
\end{proof}
 As a byproduct, we also obtain the discrete-in-space stochastic maximal regularity for the spatial semi-discretization 
\begin{equation}\label{eq:spatial_scheme}
\begin{cases}
    du_h(t) + (\lambda_0+A_h) u_h(t) \, dt =  P_h g(t) \, dW(t), \quad t \in (0,T),
    \\
    u_h(0) = 0.
\end{cases}
\end{equation}

\begin{theorem}[Discrete-in-space SMR]\label{thm:space DSMR}
 Let $q \in [2,\infty)$ and $p \in (2,\infty)$, where we allow $p=2$ if $q=2$. Suppose that Assumption \ref{main assumptions} holds for the principal part $A^\#$ of $A$. Then there exists a constant $C>0$, independent of $h$, such that for every $g \in L^p_{\mathbb F}(\Omega \times (0,T); \gamma(H, X^q_{0}))$, the discrete mild solution  $u_h$ of \eqref{eq:spatial_scheme} satisfies
 \begin{equation}\label{ineq:space DSMR with discrete trace space}
 \begin{aligned}
        \Big (\E \sup_{0\le t \le T} \|u_h(t)\|^p_{(X^q_{0,h},X^q_{1,h})_{1/2-1/p,p}} \Big )^{1/p}&+  \|  u_h \|_{L^p(\Omega \times (0,T);X^q_{1/2,h})} 
        \\
        &\le C \|P_hg\|_{L^p(\Omega \times (0,T); \gamma(H, X^q_{0,h}))} 
 \end{aligned}
    \end{equation}
    and 
    \begin{equation}\label{ineq:space DSMR}
   \begin{aligned}
       \Big ( \E \sup_{0\le t \le T} \|u_h(t)\|^p_{(X^q_{0},X^q_{1})_{1/2-1/p,p}} \Big )^{1/p} &+  \|  u_h \|_{L^p(\Omega \times (0,T);X^q_{1/2})} 
       \\
       &\le C \|P_hg\|_{L^p(\Omega \times (0,T); \gamma(H, X^q_{0,h}))} .
   \end{aligned}
\end{equation}
Furthermore, the right-hand sides of both estimates can be replaced by  $C\|g\|_{L^p(\Omega \times (0,T); \gamma(H, X^q_{0}))}$.  
\end{theorem}
\begin{proof}
By Theorem \ref{thm:H-calc for Ah}, $\lambda_0 + A_h$ admits a bounded $H^\infty$-calculus with a constant independent of $h$. Hence, by Theorem \ref{thm:SMR in Lq spaces}  there is a $C>0$ independent of $h$ such that 
$$ \Big ( \E \sup_{0\le t \le T} \|u_h(t)\|^p_{(X^q_{0,h},X^q_{1,h})_{1/2-1/p,p}} \Big )^{1/p} +  \|  u_h \|_{L^p(\Omega \times (0,T);X^q_{1/2,h})}   \le C \|P_hg\|_{L^p(\Omega \times (0,T); \gamma(H, X^q_{0,h}))} .$$
This proves \eqref{ineq:space DSMR with discrete trace space}. The inequality \eqref{ineq:space DSMR} follows from \eqref{ineq:space DSMR with discrete trace space} and the norm equivalences \eqref {Xah norm equiv Xa norm with lower-order}, \eqref{ineq: equivalence of real interpolation for trace space with lower-order}. The final remark follows from \eqref{ineq:ideal property}.
\end{proof}

\begin{remark}[Deterministic discrete-in-space maximal regularity] \label{remark:discrete in space MR}
Let $T \in (0,\infty]$ and let $p,q\in (1,\infty)$. For $h>0$ and $f\in L^p(0,T;X^q_0)$ let 
\begin{equation}\label{def:mild solution deterministic discrete space}
    u_h(t) = \int_0^t \exp{-(t-s)(\lambda_0+A_h)} P_h f(s)\, ds, \quad t \in (0,T).
\end{equation}
Suppose that Assumption \ref{main assumptions} holds for the principal part of $A$. Then, by Theorem \ref{thm:H-calc for Ah}, $\lambda_0+A_h$ has a bounded $H^\infty$-calculus with a constant independent of $h$. Consequently, by \cite[Corollary 17.3.6]{analysis_volume_3} and the stability of $P_h$, the mild solution $u_h$ given in \eqref{def:mild solution deterministic discrete space} satisfies 
\begin{equation}\label{ineq:discrete space MR}
    \|\partial_t u_h\|_{L^p(0,T;X^q_{0,h})} + \|  u_h\|_{L^p(0,T;X^q_{1,h})} \le C \|P_h f\|_{L^p(0,T;X^q_{0,h})}\le C\|f\|_{L^p(0,T;X^q_0)} ,
\end{equation}
where $C>0$ is independent of $h$.

The estimate \eqref{ineq:discrete space MR} is well-established in the literature. Geissert~\cite{Geissert-applications-of-DMR} gave a proof using kernel estimates that typically require $C^{2,\alpha}$ coefficients $a_{ij}$. Li~\cite{BuyangLi2015} proved  \eqref{ineq:discrete space MR} in the case of Neumann boundary conditions with $W^{1,\infty}$ coefficients. For the corresponding results in the case of polyhedral domains, we refer the reader to \cite{Buyang_Li2019} and the references therein.

\end{remark}

\section{Applications}\label{sec:applications}
We apply the preceding fully discrete stochastic maximal regularity results of Section \ref{sec:fully DSMR} to the
three-dimensional stochastic Allen--Cahn equation. We extend the
convergence analysis of
\cite{Zhou_Li_2025_allen_cahn,LiZhou2025_fully_discrete_AC}
to variable diffusion coefficients $(a_{ij})_{i,j=1}^3$.

\subsection{The equation and the main convergence result}
Let $\O\subset\R^3$ be a bounded, convex domain of class $C^2$,
and let $T\in(0,\infty)$. We consider the three-dimensional stochastic Allen--Cahn equation on  $\O$ with Dirichlet boundary conditions given by 
\begin{equation}\label{eq:AC}
\begin{cases}
du
=
\bigl[
\nabla\cdot(a\nabla u)+u-u^3
\bigr]\,dt
+G(u) \, dW_{\ell^2}(t)
& \text{in } (0,T]\times\mathcal O,\\
u=0,
& \text{on } (0,T]\times\partial\mathcal O,\\
u(0)=u_0,
& \text{in } \mathcal O.
\end{cases}
\end{equation}
Here, $a=(a_{ij})_{i,j=1}^3$ is real-valued, symmetric and uniformly
elliptic, with $a_{ij}\in W^{1,\infty}(\O)$, and
$u_0\in W^{1,\infty}_0(\O)\cap W^{2,\infty}(\O)$ is deterministic.

The noise coefficient is defined by
\[
 (G(v)e_k)(x)=g_k(x,v(x)),\qquad k\in\N,
\]
where $(e_k)$ is the standard basis of $\ell^2$ and
$g_k\in C^1(\overline\O\times\R)$ satisfy
\begin{align*}
&\sum_{n\in \N} |g_n(x,0)|^2=0 , && x \in \partial \O, 
\\
&\sum_{n \in \N} (|g_n(x,y)|^2+|\nabla_x g_n(x,y)|^2) \le C (1+|y|^2), && x \in \O, \quad y \in \R,
\\
&\sum_{n \in \N} |\nabla_y g_n(x,y)|^2 \le C , && x \in \O, \quad y \in \R.
\end{align*}
In particular, for $q\in[2,\infty)$,
\begin{equation}\label{ineq:AC noise estimates}
\begin{aligned}
 \|G(v)\|_{\gamma(\ell^2,X^q_0)}
    &\le C(1+\|v\|_{X^q_0}),\\
 \|G(v)-G(w)\|_{\gamma(\ell^2,X^q_0)}
    &\le C\|v-w\|_{X^q_0},\\
 \|G(v)\|_{\gamma(\ell^2,X^q_{1/2})}
    &\le C(1+\|v\|_{X^q_{1/2}}),
\end{aligned}
\end{equation}
cf.\,\cite[Lemma~4.5]{Zhou_Li_2025_allen_cahn}.

Throughout this section, $A\coloneq -\nabla\cdot(a\nabla)$ with Dirichlet
boundary conditions, and $A_h$ denotes its discretization given by 
\begin{equation}\label{eq:def of Ah applications}
 (A_hv_h,w_h)_{L^2}=(a\nabla v_h,\nabla w_h)_{L^2},
 \qquad v_h,w_h\in S_h.
\end{equation}
We recall that 
$X^q_0=L^q(\O)$, $X^q_1=W^{2,q} \cap W^{1,q}_0$,  and
$\|v_h\|_{X^q_{\alpha,h}}=\|A_h^\alpha v_h\|_{L^q}$. 

\begin{proposition}\label{prop:GWP}
    There exists a unique mild solution $u$ to equation \eqref{eq:AC} such that 
\begin{equation} \label{eq:GWP regularity}
    u \in L^p(\Omega \times (0,T);X^q_1)\cap L^p(\Omega;C([0,T];X^q_{1-\eps})),
\end{equation}
for every $p \in (2,\infty)$, $q \in [2,\infty)$ and $\eps>0$. 
\end{proposition} 
\begin{proof}
The energy argument in
\cite[Theorem~3.1 and Lemma~4.6]{Zhou_Li_2025_allen_cahn}
extends to $A$, since
\[
 \int_\O |v|^{q-2}v\,\nabla\cdot(a\nabla v)\,dx
 =-(q-1)\int_\O |v|^{q-2}a\nabla v\cdot\nabla v\,dx\le0.
\] 
It yields a unique mild solution with
\begin{equation}\label{eq:original regularity of u}
    u \in L^p(\Omega \times(0,T);L^q), \qquad p \in (2,\infty),\quad  q \in [2,\infty).
\end{equation}
By standard (stochastic) maximal regularity arguments, using \eqref{ineq:AC noise estimates} and the initial bootstrap \eqref{eq:original regularity of u},
we obtain \eqref{eq:GWP regularity} (cf.\,\cite[Section 4.2]{Zhou_Li_2025_allen_cahn}).
\end{proof}

Let $N\in\N$, $\tau=T/N$ and $t_n=n\tau$. We consider the  discretization of \eqref{eq:AC}
 given by 
\begin{equation}\label{eq:fully discrete scheme for AC}
\begin{cases}
 U_{n+1}^h-U_n^h+\tau A_hU_{n+1}^h
 =\tau\bigl(U_n^h-P_h[(U_{n+1}^h)^3]\bigr)
   +P_h\displaystyle\int_{t_n}^{t_{n+1}}
       G(U_n^h)\,dW_{\ell^2}(s),\\
 U_0^h=P_hu_0,
\end{cases}
\end{equation}
for $n=0,\ldots,N-1$.

\begin{theorem}\label{thm:convergence result}
  Let $u$ be the mild solution of equation \eqref{eq:AC} given by Proposition \ref{prop:GWP} and let $(U_n^h)$ be the approximation scheme given in \eqref{eq:fully discrete scheme for AC}. Suppose that $\tau \le h^{1+\delta}<1$ for some $\delta>0$.  Then, 
    \begin{equation}
        \Bigl ( \E \max_{0 \le n \le N} \|U_n^h - u(t_n)\|_{L^q}^p \Bigr)^{1/p} \le C (h^{2-\eps} +\tau^{1/2}),
    \end{equation}
    for every $p \in (2,\infty)$, $q \in [2,\infty]$ and $\epsilon>0$.  Here, $C>0$ is independent of $\tau$ and $h$. 
\end{theorem}
As remarked in \cite{LiZhou2025_fully_discrete_AC}, the CFL condition
$\tau\le h^{11/9}$ imposed in their Theorem~3.1 is not optimal.
The argument below yields the condition $\tau\le h^{1+\delta}$
for every fixed $\delta>0$, which weakens the original restriction
when $0<\delta<2/9$.

\subsection{Auxiliary estimates}
Before presenting the proof of Theorem \ref{thm:convergence result}, we first record some auxiliary results. 
Recall that, by Theorem \ref{theorem:Hcalc for Ah}, the operator $A_h$ has a bounded $H^\infty$-calculus with a constant $C>0$ that is independent of $h$.

By Theorem \ref{thm:fully discrete SMR}, we immediately obtain the following:

\begin{lemma}[Fully discrete stochastic maximal regularity]\label{lemma:fully DSMR}
    Let $q \in [2,\infty)$ and $p \in (2,\infty)$, where we allow $p=2$ if $q=2$, and let $(Y_n^h)$ be given by 
    \begin{equation*}
    \begin{cases}
        Y_{n+1}^h \coloneq Y_n^h - \tau A_h Y_{n+1}^h + P_h \int_{t_n}^{t_{n+1}} g(s)  \, dW_{\ell^2}(s),\qquad 0 \le n \le N-1, 
        \\
        Y_0^h =0,
    \end{cases}
\end{equation*}
    where $g \in L^p_{\mathbb F}(\Omega \times(0,T); \gamma(\ell^2,X^q_0))$.
    Then, there is $C>0$, independent of $\tau$ and $h$, such that 
    \begin{equation}\label{ineq:space-time DSMR with discrete trace space application}
      \begin{aligned}
          \Big (  \E \sup_{n=1,\dots, N} \| Y^h_n \|^p_{(X^q_{0,h},X^q_{1,h})_{1/2-1/p,p}} \Big)^{1/p}& +  \Big( \E \sum_{n=1}^{N-1} \tau  \|  Y^h_n \|^p_{X^q_{1/2,h}} \Big)^{1/p} 
          \\
          & \le C \|P_h g\|_{L^p(\Omega \times (0,T); \gamma(\ell^2, X^q_{0,h}))} .
      \end{aligned}
\end{equation}
\end{lemma}

We also recall the following deterministic version of Lemma \ref{lemma:fully DSMR}:

\begin{lemma}[Fully discrete deterministic maximal regularity]\label{lemma:fully DMR}
Let $p,q \in (1,\infty)$ and let $(Y_n^h)$ be given by 
\begin{equation*}
    \begin{cases}
        Y_{n+1}^h \coloneq Y_n^h - \tau A_h Y_{n+1}^h + P_h \int_{t_n}^{t_{n+1}} f(s) \, ds ,\qquad 0 \le n \le N-1, 
        \\
        Y_0^h =0,
    \end{cases}
\end{equation*}
where $f \in L^p(0,T;X^q_0)$. Then, there is $C>0$, independent of $\tau$ and $h$, such that 
\begin{equation}\label{ineq:fully discrete deterministic MR}
\begin{aligned}
   \sup_{n=1,\dots,N} \|Y_n^h\|_{(X^q_{0,h},X^q_{1,h})_{1-1/p,p}} &+  \Bigl (\sum_{n=1}^{N-1} \tau \|Y_n^h\|^p_{X^q_{1,h}} \Bigr )^{1/p} + \Bigl (\sum_{n=1}^{N-1} \tau \|D_\tau Y_n^h\|^p_{X^q_{0,h}} \Bigr )^{1/p}  
    \\
    &
    \le C \|f\|_{L^p(0,T;X^q_0)}.
\end{aligned}
\end{equation}
Here, $D_\tau Y_n\coloneq (Y_n-Y_{n-1})/\tau $ denotes the backward difference. 
\end{lemma}
\begin{proof}
The estimate 
\begin{equation*}
    \Bigl (\sum_{n=1}^{N-1} \tau \|Y_n^h\|^p_{X^q_{1,h}} \Bigr )^{1/p}  +  \Bigl (\sum_{n=1}^{N-1} \tau \|D_\tau Y_n^h\|^p_{X^q_{0,h}} \Bigr )^{1/p}  
    \le C \|f\|_{L^p(0,T;X^q_0)}
\end{equation*}
can be found in \cite[Theorems 3.1 and 6.1]{Kovacs_Astable}. The maximal estimate 
\begin{equation*}
    \sup_{n=1,\dots,N} \|Y_n^h\|_{(X^q_{0,h},X^q_{1,h})_{1-1/p,p}}\le C \|f\|_{L^p(0,T;X^q_0)}
\end{equation*}
is proved by reducing to the continuous-time setting as in \cite[Remark 6.2]{DSMR} and using the classical trace embedding (see \cite[Theorem L.4.1]{analysis_volume_3}). 
\end{proof}

For future reference, we also recall the following properties of the $L^2$ and the Ritz projections.
\begin{lemma}\label{lemma:prelimin estimates}
    The following estimates hold with a constant $C>0$ that is independent of $h$: 
    \begin{enumerate}[(i)]
        \item For $q \in (1,\infty)$ and $\alpha \in [0,1]$,
    \begin{align}
        \|I-P_h\|_{\L(X^q_\alpha,X^q_0)} &\le C h^{2\alpha},  \label{Convergence estimate for I-Ph interpolation}
    \intertext{and}
\label{Ph stable Xqbeta to Xqbetah with lower-order application}
    \|P_h\|_{\L( X^q_\alpha ,  X^q_{\alpha,h})} &\le C.
\end{align}

    \item For $\alpha \in [0,1/2]$, 
\begin{equation}\label{Xah norm equiv Xa norm with lower-order application}
    C^{-1} \|u_h\|_{ X^q_{\alpha}} \le   \|u_h\|_{ X^q_{\alpha,h}} \le C\|u_h\|_{ X^q_{\alpha}} , \quad u_h \in S_h.
\end{equation} 
    \item For $q \in [2,\infty)$ and $\alpha \in [1/2,1]$, 
\begin{equation}
    \|I-R_h\|_{\L(X^q_\alpha,X^q_0)} \le C h^{2\alpha}.\label{Convergence estimate for I-Rh interpolation}
\end{equation}
When $\alpha=1$,  \eqref{Convergence estimate for I-Rh interpolation} holds for any $q \in (1,\infty)$. 
\end{enumerate}
\end{lemma}
\begin{proof}
The estimates \eqref{Ph stable Xqbeta to Xqbetah with lower-order application} and \eqref{Xah norm equiv Xa norm with lower-order application} follow directly from Corollary \ref{cor: Xah norm equiv Xa norm with lower-order terms}, whereas \eqref{Convergence estimate for I-Ph interpolation} follows from interpolating between \eqref{stability of Ph on X0} and \eqref{Convergence estimate for I-Ph}. Similarly, \eqref{Convergence estimate for I-Rh interpolation} follows by interpolating between \eqref{Convergence estimate for I-Rh} and \eqref{Convergence estimate for I-Rh lower order}. 
\end{proof}

Finally, one has the following Sobolev embeddings for the discrete spaces $X^q_{\alpha,h}$.

\begin{lemma}[Sobolev embeddings]\label{lemma:disctete Sob emb}
Let $q \in (1,\infty)$. 
The following estimates hold with a constant $C>0$ that is independent of $h$:
\begin{enumerate}[(i)]
    \item For $\alpha \in (0,1)$ and $q \in (\tfrac 3{2\alpha},\infty)\cap  [2,\infty]$,
\begin{equation*}
    \|u_h\|_{L^\infty} \le C \|u_h\|_{X^q_{\alpha,h}}, \qquad u_h \in S_h.
\end{equation*}
\item If $3<\beta<\infty$ and $\frac{3\beta}{3+2\beta}\le q <\infty$ then 
\begin{equation*}
    \|u_h\|_{L^\beta} \le C \|u_h\|_{X^q_{1,h}}, \qquad u_h \in S_h.
\end{equation*}
If $q>3/2$, then 
\begin{equation*}
    \|u_h\|_{L^\infty} \le C \|u_h\|_{X^q_{1,h}}, \qquad u_h \in S_h.
\end{equation*}
\item If $\frac 32<\beta<\infty$ and $\frac{3\beta}{3+\beta}\le q < \infty$, then 
$$    \|u_h\|_{L^\beta} \le C \|u_h\|_{X^q_{1/2,h}}, \qquad u_h \in S_h.$$
\end{enumerate}
    
\end{lemma}
\begin{proof}
    One can apply the proof of   \cite[Lemma 4.6]{LiZhou2025_fully_discrete_AC} mutatis mutandis, using  Lemma \ref{lemma:prelimin estimates}.
\end{proof}

We now extend the convergence estimates of \cite[Theorem~3.3]{Zhou_Li_2025_allen_cahn} for the spatial semi-discretization of \eqref{eq:AC} given by 
\begin{equation}\label{eq:AC semidiscretization}
\begin{cases}
 du_h+A_hu_h\,dt=(u_h-P_hu_h^3)\,dt
                 +P_hG(u_h)\,dW_{\ell^2}(t), \quad t \in (0,T),
                 \\
 u_h(0)=P_hu_0.
\end{cases}
\end{equation}

\begin{proposition}[Convergence estimate for the spatial semidiscretization]
\label{prop:AC spatial convergence}
Let $u$ be the mild solution of equation \eqref{eq:AC} given by Proposition \ref{prop:GWP} and let $u_h$ be as in \eqref{eq:AC semidiscretization}. Then,  
\begin{align}
 \|u-u_h\|_{L^p(\Omega\times(0,T);L^q)}
 &\le Ch^2,
 \label{ineq:spatial convergence 1}\\
 \|u-u_h\|_{L^p(\Omega;C([0,T];L^\infty))}
 &\le Ch^{2-\eps},
 \label{ineq:spatial convergence 2}
\end{align}
for every $p,q\in(2,\infty)$ and $\eps>0$. 
Here, $C>0$ is independent of $h$.
\end{proposition}
\begin{proof}
We apply the argument of
\cite[Theorem~3.3]{Zhou_Li_2025_allen_cahn}.
The required projection estimates are supplied by
Lemma~\ref{lemma:prelimin estimates}. The estimates for the Ritz projection and the uniform
sectoriality of $A_h$ give the semigroup approximation estimates in
\cite[Lemma~4.9]{Zhou_Li_2025_allen_cahn}.
The required discrete-in-space deterministic and stochastic maximal
regularity follow from
Remark~\ref{remark:discrete in space MR} and Theorem~\ref{thm:space DSMR}, respectively.
The nonlinear energy estimates used to derive the $L^2$-error estimates remain valid by the symmetry and ellipticity of the coefficient matrix $a=(a_{ij})_{i,j=1}^3$. Thus, together with the discrete Sobolev embeddings in Lemma~\ref{lemma:disctete Sob emb}, the proof of \cite[Theorem~3.3]{Zhou_Li_2025_allen_cahn} applies mutatis mutandis.
\end{proof}

\subsection{Proof of Theorem \ref{thm:convergence result}}
 We are now ready to present the proof of Theorem \ref{thm:convergence result}.
\begin{proof}[Proof of Theorem \ref{thm:convergence result}]
By Proposition \ref{prop:AC spatial convergence}, it suffices to prove that, for every $p \in (2,\infty)$,
\begin{equation}\label{ineq:sufficient estimate to prove}
    \bigl (  \E  \max_{0\le n \le N} \|U_n^h-u_h(t_n)\|_{L^\infty}^p \bigr )^{1/p} \le C \tau^{1/2}.
\end{equation}
We now establish \eqref{ineq:sufficient estimate to prove} by adapting
the proof of \cite[Theorem~3.1]{LiZhou2025_fully_discrete_AC}.
We divide the proof into six steps, corresponding to those in
\cite[Section~4.2]{LiZhou2025_fully_discrete_AC}.
Since the arguments there are formulated for the Dirichlet Laplacian,
in Steps 1, 2, and 5 we indicate which estimates established above
for $A$ and $A_h$ allow us to carry over the corresponding arguments,
and omit the repeated details.

In Steps 3, 4, and 6, we give alternative arguments.
In Step 3, we prove a stronger $L^2$-error estimate by using directly
the monotonicity of the cubic nonlinearity and the positivity of $A_h$.
This avoids the auxiliary decomposition used in
\cite[Section~4.2.3]{LiZhou2025_fully_discrete_AC}.
In Step 4, our argument allows for the more general CFL condition
$\tau \le h^{1+\delta}$ in place of the condition $\tau \le h^{11/9}$
imposed in \cite[Theorem~3.1]{LiZhou2025_fully_discrete_AC}.
Finally, in Step 6, we retain the error decomposition of
\cite[Section~4.2.5]{LiZhou2025_fully_discrete_AC}, but replace the
resolvent estimates and convolution bounds used in
\cite[Section~4.2.6]{LiZhou2025_fully_discrete_AC} with the deterministic
and stochastic maximal estimates established above.
These yield the temporal error estimate directly in the natural
trace space $(X_{0,h}^q,X_{1,h}^q)_{1/2-1/p,p}$, from which
\eqref{ineq:sufficient estimate to prove} follows by a Sobolev embedding.

\medskip
\noindent
\emph{Step 1: stability estimates for $u_h$.}
For every $p,q \in (2,\infty)$, 
\begin{align*}
    &\|u_h\|_{L^p(\Omega \times (0,T);X^q_{1,h})}  \le C
    \\
    &\|u_h\|_{L^p(\Omega ; C([0,T];L^\infty))}  \le C
    \\
    &\|P_h u_h^3\|_{L^p(\Omega ; C([0,T];X^q_{1/2,h}))}+\|P_h G(u_h)\|_{L^p(\Omega ; C([0,T]; \gamma(\ell^2,X^q_{1/2,h})))}  \le C.
\end{align*}
This is proved using Lemmas \ref{lemma:prelimin estimates} and \ref{lemma:disctete Sob emb}, and the regularity of $u$ in \eqref{eq:GWP regularity} given by Proposition \ref{prop:GWP}.

\medskip
\noindent
\emph{Step 2: temporal increment estimates for $u_h$.} For every $p,q \in (2,\infty)$, 
\begin{align}
\mathbb{E} \biggl[ \sum_{n=0}^{N-1} \int_{t_n}^{t_{n+1}} \|u_h(t) - u_h(t_n)\|_{X^q_0}^p \,\mathrm{d}t \biggr] &\le C\tau^{p/2},  \label{ineq:Step 4 1}
\\
\mathbb{E} \biggl[ \sum_{n=0}^{N-1} \int_{t_n}^{t_{n+1}} \|u_h(t) - u_h(t_{n+1})\|_{X^q_0}^p \,\mathrm{d}t \biggr] &\leq C\tau^{p/2}, \label{ineq:Step 4 2}
 \\
\mathbb{E} \biggl[ \sum_{n=0}^{N-1} \int_{t_n}^{t_{n+1}} \|u_h(t) - u_h(t_{n+1})\|_{X^q_{1/2,h}}^p \,\mathrm{d}t \biggr] &\le C\tau^{p/2},\label{ineq:Step 4 3}
 \\
\mathbb{E} \biggl[ \sum_{n=0}^{N-1} \int_{t_n}^{t_{n+1}} \|u_h^3(t) - u_h^3(t_{n+1})\|_{X^q_0}^p \,\mathrm{d}t \biggr] &\leq C\tau^{p/2}. \label{ineq:Step 4 4}
\end{align}
This follows from the stability estimates in Step 1 and standard sectoriality estimates for $A_h$. 

\medskip
\noindent
\emph{Step 3: $L^2$-error estimate.}
We show that, for every $p \in [2,\infty)$, 
\begin{equation}\label{ineq:L2 temporal error}
    \max_{0\le n\le N}
    \E \|U_n^h-u_h(t_n)\|_{L^2}^p
    \le C\tau^{p/2}.
\end{equation}
In particular, this implies the weaker estimate derived in \cite[Equation (4.20)]{LiZhou2025_fully_discrete_AC}, namely, 
\begin{equation}\label{ineq:L2 temporal error sum}
    \E \biggl[
    \sum_{n=0}^N \tau 
    \|U_n^h-u_h(t_n)\|_{L^2}^p
    \biggr]
    \le C\tau^{p/2}.
\end{equation}

For notational brevity, we define
\begin{align*}
    w_n&\coloneq
    P_h\int_{t_n}^{t_{n+1}}
    \bigl(G(U_n^h)-G(u_h(t))\bigr)\,dW(t),
    \\
    e_n &\coloneq U_n^h - u_h(t_n), \qquad v_n \coloneq (1+\tau)e_n + w_n,
    \\
    z_n &\coloneq \int_{t_n}^{t_{n+1}}
        \bigl(u_h(t_{n+1})-u_h(t)\bigr)\,dt, 
        \\
    f_n & \coloneq   \int_{t_n}^{t_{n+1}}
        \bigl(u_h(t_n)-u_h(t)\bigr)\,dt
        +P_h\int_{t_n}^{t_{n+1}}
        \bigl(u_h^3(t)-u_h^3(t_{n+1})\bigr)\,dt.
\end{align*}
Subtracting the spatial semi-discretization scheme \eqref{eq:AC semidiscretization} from the fully discrete one \eqref{eq:fully discrete scheme for AC} yields 
\begin{align}
    e_{n+1}+\tau A_he_{n+1}
    +\tau P_h\bigl((U_{n+1}^h)^3
    -u_h^3(t_{n+1})\bigr)
    =v_n-A_hz_n+f_n.
    \label{eq:L2 error recurrence}
\end{align}

We test  
\eqref{eq:L2 error recurrence} with $e_{n+1}$.
Since $P_h$ is the $L^2$-orthogonal projection and
the cubic nonlinearity is monotone, we have
\begin{align*}
    &\bigl(
    P_h((U_{n+1}^h)^3-u_h^3(t_{n+1})),
    e_{n+1}\bigr)_{L^2}=
    \int_{\O}
    \bigl((U_{n+1}^h)^3-u_h^3(t_{n+1})\bigr)
    \bigl(U_{n+1}^h-u_h(t_{n+1})\bigr)\,dx
    \ge 0.
\end{align*}
Moreover, since $A_h$ is positive and self-adjoint
on $X^2_{0,h}$, it follows that
\begin{align*}
    (A_he_{n+1},e_{n+1})_{L^2}
    =\|A_h^{1/2}e_{n+1}\|_{L^2}^2.
\end{align*}
Consequently,
\begin{align*}
    &\|e_{n+1}\|_{L^2}^2
    +\tau\|A_h^{1/2}e_{n+1}\|_{L^2}^2 
    \le
    (v_n,e_{n+1})_{L^2}
    -(A_hz_n,e_{n+1})_{L^2}
    +(f_n,e_{n+1})_{L^2}.
\end{align*}
By Cauchy--Schwarz and Young's inequalities,
\begin{align*}
    (v_n,e_{n+1})_{L^2}
    &\le \frac12\|v_n\|_{L^2}^2
    +\frac12\|e_{n+1}\|_{L^2}^2,\\
    |(A_hz_n,e_{n+1})_{L^2}|
    &\le \frac{1}{2\tau}
    \|A_h^{1/2}z_n\|_{L^2}^2
    +\frac{\tau}{2}
    \|A_h^{1/2}e_{n+1}\|_{L^2}^2,\\
    |(f_n,e_{n+1})_{L^2}|
    &\le \frac{1}{\tau}\|f_n\|_{L^2}^2
    +\frac{\tau}{4}\|e_{n+1}\|_{L^2}^2.
\end{align*}
Combining the preceding estimates gives
\begin{align*}
    &\Bigl(1-\frac{\tau}{2}\Bigr)
    \|e_{n+1}\|_{L^2}^2
    +\tau\|A_h^{1/2}e_{n+1}\|_{L^2}^2
    \le
    \|v_n\|_{L^2}^2
    +\frac1\tau\|A_h^{1/2}z_n\|_{L^2}^2
    +\frac2\tau\|f_n\|_{L^2}^2.
\end{align*}
Since $(1-\tau/2)^{-1}\le 1+\tau$, we obtain
\begin{align}
    \|e_{n+1}\|_{L^2}^2
    \le{}&(1+\tau)\|v_n\|_{L^2}^2
    +\frac{C}{\tau}
    \|A_h^{1/2}z_n\|_{L^2}^2
    +\frac{C}{\tau}\|f_n\|_{L^2}^2.
    \label{ineq:L2 energy preliminary}
\end{align}
Moreover, by Cauchy--Schwarz and the definition
of $X^2_{1/2,h}$, 
\begin{align*}
    \frac1\tau\|A_h^{1/2}z_n\|_{L^2}^2
    &\le \int_{t_n}^{t_{n+1}}
    \|u_h(t_{n+1})-u_h(t)\|_{X^2_{1/2,h}}^2\,dt.
\end{align*}
Similarly, using the $L^2$-contractivity of $P_h$,
we obtain
\begin{align*}
    \frac1\tau\|f_n\|_{L^2}^2
    \le 2\int_{t_n}^{t_{n+1}}
    \bigl(\|u_h(t_n)-u_h(t)\|_{L^2}^2
    +\|u_h^3(t)-u_h^3(t_{n+1})\|_{L^2}^2\bigr)\,dt.
\end{align*}
For $m\ge 1,$ we let
\begin{align*}
    K_n(m) \coloneq \int_{t_n}^{t_{n+1}}
    \bigl(\|u_h(t_{n+1})-u_h(t)\|_{X^2_{1/2,h}}^{2m}+\|u_h(t_n)-u_h(t)\|_{L^2}^{2m}
    +\|u_h^3(t)-u_h^3(t_{n+1})\|_{L^2}^{2m}\bigr)\,dt.
\end{align*}
It follows from \eqref{ineq:L2 energy preliminary} and the calculations above
that
\begin{equation}\label{ineq:L2 energy}
    \|e_{n+1}\|_{L^2}^2
    \le (1+C\tau)\|v_n\|_{L^2}^2+CK_n(1).
\end{equation}
We now fix $m\ge 1$. By Young's
inequality, 
\begin{align*}
    (x+y)^m
    \le (1+C\tau)x^m
    +C\tau^{1-m}y^m, \qquad x,y \ge 0.
\end{align*}
Applying this to \eqref{ineq:L2 energy},
we deduce that
\begin{equation}\label{ineq:L2 energy moments}
    \E\|e_{n+1}\|_{L^2}^{2m}
    \le (1+C\tau)\E\|v_n\|_{L^2}^{2m}
    +C\E K_n(m).
\end{equation}
Here, we also used H\"older's inequality to bound
$\tau^{1-m}K_n(1)^m \le 3^{m-1}K_n(m)$.

We now estimate $\E\|v_n\|_{L^2}^{2m}$.
Recall that $v_n=(1+\tau)e_n+w_n$. We note the elementary inequality 
\begin{equation}\label{ineq:elementary inequality for L2 functions}
    \|f+g\|_{L^2}^{2m} \le \|f\|_{L^2}^{2m} +2m\|f\|_{L^2}^{2m-2} (f,g)_{L^2} + C \|f\|_{L^2}^{2m-2}\|g\|_{L^2}^2 + C \|g\|_{L^2}^{2m}, \quad f,g \in L^2.
\end{equation}
Indeed, this follows by estimating the second derivative of  $\phi(\theta)\coloneq \|f+\theta g\|_{L^2}^{2m}$ by 
$\phi''(\theta) \le C (\|f\|_{L^2}^{2m-2}\|g\|_{L^2}^2+\|g\|_{L^2}^{2m})$
and the identity
$ \phi(1)=\phi(0)+\phi'(0)+\int_0^1 (1-\theta) \phi''(\theta)\, \dd \theta.$

Applying \eqref{ineq:elementary inequality for L2 functions} to $v_n=(1+\tau)e_n +w_n$ gives 
\begin{equation}\label{ineq:estimate for bn 1}
\begin{aligned}
    \|v_n\|_{L^2}^{2m} 
    &\le
     (1+\tau)^{2m}\|e_n\|_{L^2}^{2m} +2m (1+\tau)^{2m-1} \|e_n\|_{L^2}^{2m-2}(e_n,w_n)_{L^2}
     \\
     & \phantom{\le }+ C(1+\tau)^{2m-2}\|e_n\|_{L^2}^{2m-2} \|w_n\|_{L^2}^2 + C \|w_n\|_{L^2}^{2m}.
\end{aligned}
\end{equation}
Since $e_n$ is $\mathcal F_{t_n}$-measurable
and $w_n$ is a martingale increment, we have
\begin{align*}
    \E\bigl[
    \|e_n\|_{L^2}^{2m-2}(e_n,w_n)_{L^2}
    \bigr]
    =0.
\end{align*}
Therefore, taking expectations in \eqref{ineq:estimate for bn 1} yields
\begin{equation}\label{ineq:estimate for bn 2}
   \E\|v_n\|_{L^2}^{2m} 
    \le
     (1+\tau)^{2m}\E \|e_n\|_{L^2}^{2m} + C(1+\tau)^{2m-2}\E (\|e_n\|_{L^2}^{2m-2} \|w_n\|_{L^2}^2 )+ C  \E \|w_n\|_{L^2}^{2m}.
\end{equation}
By Young's inequality,
\begin{align*}
    \|e_n\|_{L^2}^{2m-2}\|w_n\|_{L^2}^2
    \le C\tau\|e_n\|_{L^2}^{2m}
    +C\tau^{1-m}\|w_n\|_{L^2}^{2m}.
\end{align*}
Since $(1+\tau)^{2m}\le1+C\tau$, we conclude that
\begin{equation}\label{ineq:L2 bn moments}
    \E\|v_n\|_{L^2}^{2m}
    \le (1+C\tau)\E\|e_n\|_{L^2}^{2m}
    +C\tau^{1-m}\E\|w_n\|_{L^2}^{2m}.
\end{equation}

Next, the Burkholder--Davis--Gundy inequality,
the $L^2$-contractivity of $P_h$, and the
Lipschitz continuity of $G$ imply
\begin{align*}
    \E\|w_n\|_{L^2}^{2m}
    &\le C\E\biggl(
    \int_{t_n}^{t_{n+1}}
    \|G(U_n^h)-G(u_h(t))\|_{\gamma(\ell^2,L^2)}^2
    \,dt\biggr)^m
    \le C\E\biggl(
    \int_{t_n}^{t_{n+1}}
    \|U_n^h-u_h(t)\|_{L^2}^2\,dt
    \biggr)^m\\
    &\le C\tau^{m-1}
    \int_{t_n}^{t_{n+1}}
    \E\|U_n^h-u_h(t)\|_{L^2}^{2m}\,dt.
\end{align*}
In the last step, we used H\"older's inequality. Since
$
    U_n^h-u_h(t)=e_n+u_h(t_n)-u_h(t),
$
we obtain
\begin{equation}
\begin{aligned}
    \tau^{1-m}\E\|w_n\|_{L^2}^{2m}
    &\le  C\tau\E\|e_n\|_{L^2}^{2m}
    +C\int_{t_n}^{t_{n+1}}
    \E\|u_h(t_n)-u_h(t)\|_{L^2}^{2m}\,dt 
    \\
    & \le C\tau\E\|e_n\|_{L^2}^{2m} + C \E K_n(m).
\end{aligned}
    \label{ineq:L2 wn moments}
    \end{equation}
Combining \eqref{ineq:L2 bn moments} and
\eqref{ineq:L2 wn moments} yields
\begin{align}
    \E\|v_n\|_{L^2}^{2m}
    \le{}&(1+C\tau)
    \E\|e_n\|_{L^2}^{2m}
    +C \E K_n(m)
    \label{ineq:L2 bn final}
\end{align}

Substituting  \eqref{ineq:L2 bn final} into
\eqref{ineq:L2 energy moments}, we arrive at
\begin{align}
    \E\|e_{n+1}\|_{L^2}^{2m}
    \le(1+C\tau)
    \E\|e_n\|_{L^2}^{2m}
    +C \E K_n(m)
    \label{ineq:L2 error recursive}
\end{align}
By \eqref{ineq:Step 4 1}--\eqref{ineq:Step 4 4}, we have 
\begin{equation}\label{ineq:estimate for Knm}
      \E\sum_{n=0}^{N-1}
    K_n(m)
    \le C\tau^m.
\end{equation}
Therefore, by \eqref{ineq:estimate for Knm}, the discrete Gronwall
inequality and noting that $e_0=0$, we obtain
\begin{equation*}
    \max_{0\le n\le N}
    \E\|e_n\|_{L^2}^{2m}
    \le C\tau^m,\qquad m\in\N.
\end{equation*}
This implies \eqref{ineq:L2 temporal error}. 

\medskip
\noindent
\emph{Step 4: $L^\infty$-stability estimate for $U_n^h$}. We now use the $L^2$-error estimate \eqref{ineq:L2 temporal error sum} established in Step 3, Lemma \ref{lemma:disctete Sob emb} and the fully discrete stochastic and deterministic maximal regularity estimates of Lemmas \ref{lemma:fully DSMR} and \ref{lemma:fully DMR}, to obtain
\begin{equation}\label{ineq:stability step 5}
    \E \left [ \sum_{n=1}^N \tau \|U_n^h\|_{L^\infty}^p  \right ] \le C, \qquad p \in (2,\infty).
\end{equation}
Recall that $\tau\le h^{1+\delta}<1$, where $\delta>0$. Without loss of generality, we may assume that $\delta <2/3$. Let
\[
r_0\coloneq \frac6{2-\delta} \in (3,\tfrac 92).
\]
We decompose the numerical solution as in \cite[Section 4.2.4]{LiZhou2025_fully_discrete_AC}, namely, 
\[
U_n^h=\eta_n^{(0)}+\eta_n^{(1)}+\eta_n^{(2)},
\qquad 0\le n\le N,
\]
where
\[
\begin{aligned}
 \eta_{n+1}^{(0)}-\eta_n^{(0)}+\tau A_h\eta_{n+1}^{(0)}&=0,
 &\eta_0^{(0)}&=P_hu_0,\\
 \eta_{n+1}^{(1)}-\eta_n^{(1)}+\tau A_h\eta_{n+1}^{(1)}
 &=\tau\bigl(U_n^h-P_h(U_{n+1}^h)^3\bigr),
 &\eta_0^{(1)}&=0,\\
 \eta_{n+1}^{(2)}-\eta_n^{(2)}+\tau A_h\eta_{n+1}^{(2)}
 &=P_h\int_{t_n}^{t_{n+1}}G(U_n^h)\,dW_{\ell^2}(t),
 &\eta_0^{(2)}&=0.
\end{aligned}
\]
By the identity
$
\eta_n^{(0)}=(I+\tau A_h)^{-n}P_hu_0
$, the stability of $P_h$, the uniform boundedness of $(1+\tau A_h)^{-n}$, and the embedding $X^4_{1/2,h}\hookrightarrow L^\infty$, we get
\begin{equation}\label{ineq:estimate for etan0}
    \mathbb E\left[
\sum_{n=0}^{N}\tau\|\eta_n^{(0)}\|_{L^\infty}^p
\right]\le C,
\qquad p\in(2,\infty).
\end{equation}

We next obtain an initial $L^{r_0}$-stability bound. Since
$
3(\frac12-\frac1{r_0})
=\frac{1+\delta}{2},
$
the inverse estimate gives
\[
\|v_h\|_{L^{r_0}}
\le C h^{-(1+\delta)/2}\|v_h\|_{L^2},
\qquad v_h\in S_h.
\]
Consequently, the \(L^2\)-error estimate from Step 3 and the stability estimate for \(u_h\) from Step 1 yield
\[
\begin{aligned}
\mathbb E\left[
\sum_{n=0}^{N} \tau \|U_n^h\|_{L^{r_0}}^{p}
\right]
&\quad\le C \mathbb E\left[
\sum_{n=0}^{N} \tau
\left(
\|U_n^h-u_h(t_n)\|_{L^{r_0}}^{p}
+\|u_h(t_n)\|_{L^{r_0}}^{p}
\right)\right]\\
&\quad\le C h^{-p(1+\delta)/2}
\mathbb E\left[
\sum_{n=0}^{N}\tau
\|U_n^h-u_h(t_n)\|_{L^2}^{p}
\right]+C\\
&\quad\le C\left(\frac{\tau}{h^{1+\delta}}\right)^{p/2}+C
\le C.
\end{aligned}
\]

By the stability of $P_h$, the ideal property of $\gamma$-radonifying operators, the linear growth of $G$ in \eqref{ineq:AC noise estimates}, and the triangle inequality, we obtain
\[
\begin{aligned}
&\mathbb E\left[
\sum_{n=0}^{N-1}\tau 
\left(
\|U_n^h-P_h((U_{n+1}^h)^3)\|_{L^{r_0/3}}^p
+\|P_hG(U_n^h)\|_{\gamma(\ell^2,L^{r_0})}^p
\right)\right]\le
C+C\mathbb E\left[
\sum_{n=0}^{N}\tau \|U_n^h\|_{L^{r_0}}^{3p}
\right]\le C.
\end{aligned}
\]
Lemmas \ref{lemma:fully DSMR} and \ref{lemma:fully DMR} then give
\[
\mathbb E\left[
\tau\sum_{n=0}^{N}
\left(
\|\eta_n^{(1)}\|_{X^{r_0/3}_{1,h}}^p
+\|\eta_n^{(2)}\|_{X^{r_0}_{1/2,h}}^p
\right)\right]\le C.
\]
Set
\[
r_1 \coloneq \frac{3r_0}{9-2r_0} \in (3,\infty).
\]
By Lemma \ref{lemma:disctete Sob emb}, 
$
X^{r_0/3}_{1,h}\hookrightarrow L^{r_1}
$ and 
$X^{r_0}_{1/2,h}\hookrightarrow L^\infty
\hookrightarrow L^{r_1},
$
uniformly in \(h\). 
Combining these embeddings with the bound for $\eta_n^{(0)}$ in \eqref{ineq:estimate for etan0}, we conclude that
\[
\mathbb E\left[
\sum_{n=0}^{N}\tau \|U_n^h\|_{L^{r_1}}^p
\right]\le C,
\qquad p\in(2,\infty).
\]

We now iterate this calculation. Suppose that, for some \(r_k>3\),
\[
\mathbb E\left[
\sum_{n=0}^{N}\tau \|U_n^h\|_{L^{r_k}}^p
\right]\le C,
\qquad p\in(2,\infty).
\]
By the above calculations, we obtain
\[
\mathbb E\left[
\sum_{n=0}^{N} \tau 
\left(
\|\eta_n^{(1)}\|_{X^{r_k/3}_{1,h}}^p
+\|\eta_n^{(2)}\|_{X^{r_k}_{1/2,h}}^p
\right)\right]\le C.
\]

If \(3<r_k<9/2\), we let 
\[
r_{k+1}\coloneq \frac{3r_k}{9-2r_k} >r_k.
\]
The embeddings
$X^{r_k/3}_{1,h}\hookrightarrow L^{r_{k+1}}
$ and
$X^{r_k}_{1/2,h}\hookrightarrow L^\infty
\hookrightarrow L^{r_{k+1}}$, together with \eqref{ineq:estimate for etan0}, imply
\[
\mathbb E\left[
\sum_{n=0}^{N} \tau \|U_n^h\|_{L^{r_{k+1}}}^p
\right]\le C,
\qquad p\in(2,\infty).
\]

By induction,
\[
\frac1{r_k}=\frac13-\frac{3^k\delta}{6}.
\]
Since $\delta>0$, the iteration reaches an exponent
$r_k\ge9/2$ after finitely many steps. We conclude by
considering separately the cases $r_k>9/2$ and $r_k=9/2$ below.

If $r_k>9/2$, then $r_k/3>3/2$ and thus the embeddings 
$
X^{r_k/3}_{1,h}\hookrightarrow L^\infty$
and 
$X^{r_k}_{1/2,h}\hookrightarrow L^\infty$
hold. 
Combining these embeddings with the fully discrete maximal regularity estimates of Lemmas  \ref{lemma:fully DSMR} and \ref{lemma:fully DMR}, and taking into account  \eqref{ineq:estimate for etan0}, we obtain
\[
\mathbb E\left[
\sum_{n=0}^{N} \tau \|U_n^h\|_{L^\infty}^p
\right]\le C.
\]

Finally, if $r_k=9/2$, the embedding
$
X^{3/2}_{1,h}\hookrightarrow L^6
$
first gives stability in $L^6$. Repeating the forcing calculation above yields
\[
\mathbb E\left[
\sum_{n=0}^{N} \tau 
\left(
\|\eta_n^{(1)}\|_{X^2_{1,h}}^p
+\|\eta_n^{(2)}\|_{X^6_{1/2,h}}^p
\right)\right]\le C .
\]
Since both \(X^2_{1,h}\) and \(X^6_{1/2,h}\) embed uniformly into \(L^\infty\), we obtain the same conclusion. 

This proves \eqref{ineq:stability step 5}.  

\medskip
\noindent
\emph{Step 5:}  We establish the convergence estimate 
\begin{equation}\label{ineq: almost done}
    \E \sum_{n=0}^N \tau \|U_n^h -u_h(t_n)\|_{L^\infty}^p \le C \tau^{p/2}, \qquad p \in (2,\infty).
\end{equation} 
This follows from the $L^2$-estimate \eqref{ineq:L2 temporal error sum} of Step 3, the stability estimate \eqref {ineq:stability step 5} from Step 4, Lemma \ref{lemma:disctete Sob emb}, and  the fully discrete stochastic and deterministic maximal regularity estimates of Lemmas \ref{lemma:fully DSMR} and \ref{lemma:fully DMR}.

\medskip
\noindent
\emph{Step 6.} We prove the stronger estimate 
\begin{equation}\label{ineq:stronger estimate}
    \bigl ( \E \max_{0 \le n \le N} \|U_n^h-u_h(t_n)\|^p_{(X_{0,h}^q,X_{1,h}^q)_{1/2-1/p,p}} \bigr )^{1/p} \le C \tau^{1/2}, \qquad p,q \in (2,\infty).
\end{equation}
Note that, by Corollary \ref{cor:discrete vs continuous real int with lower-order terms}, the norm of $(X_{0,h}^q,X_{1,h}^q)_{1/2-1/p,p}$ is equivalent to the norm of the usual Besov space 
$${}_{D} B^{1-2/p}_{q,p} \coloneq  (X_0^q,X_1^q)_{1/2-1/p,p}$$
and the constants do not depend on $h$. Therefore,  \eqref{ineq:stronger estimate} implies \eqref{ineq:sufficient estimate to prove}. Indeed, let $p \in (2,\infty)$ be fixed. Then, \eqref{ineq:stronger estimate} with $q>\frac{3p}{p-2}$  and the Sobolev embedding ${}_D B_{q,p}^{1-2/p}  \hookrightarrow L^\infty$ give \eqref{ineq:sufficient estimate to prove}.

We now prove \eqref{ineq:stronger estimate}. We follow the decomposition in \cite[Section 4.2.5]{LiZhou2025_fully_discrete_AC} and split the error $e_n\coloneq U_n^h-u_h(t_n)$ into three parts, namely, 
\begin{equation}\label{eq:error decomposition}
    e_n= \chi_n^{(0)}+\chi_n^{(1)}+\chi_n^{(2)},
    \qquad 0\le n\le N.
\end{equation}
The three sequences are defined recursively, for $0\le n<N$, by
\begin{align}
    \chi_{n+1}^{(0)}-\chi_n^{(0)}
    +\tau A_h\chi_{n+1}^{(0)}
    &=
    \int_{t_n}^{t_{n+1}}
    \Bigl[
        -A_h\bigl(u_h(t_{n+1})-u_h(t)\bigr)
        +U_n^h-u_h(t)
    \Bigr]\,dt,
    \label{eq:error component zero}
    \\
    \chi_{n+1}^{(1)}-\chi_n^{(1)}
    +\tau A_h\chi_{n+1}^{(1)}
    &=
    P_h\int_{t_n}^{t_{n+1}}
    (u_h^3(t)-(U_{n+1}^h)^3)\,dt,
    \label{eq:error component one}
    \\
    \chi_{n+1}^{(2)}-\chi_n^{(2)}
    +\tau A_h\chi_{n+1}^{(2)}
    &=
    P_h\int_{t_n}^{t_{n+1}}
    (G(U_n^h)-G(u_h(t)))\,dW_{\ell^2}(t),
    \label{eq:error component two}
\end{align}
with
$
    \chi_0^{(0)}=\chi_0^{(1)}=\chi_0^{(2)}=0.
$
We bound the deterministic terms  $\chi_n^{(0)}$ and $\chi_n^{(1)}$ using Lemma \ref{lemma:fully DMR} and the stochastic term $\chi_n^{(2)}$ using Lemma  \ref{lemma:fully DSMR}.

Multiplying \eqref{eq:error component zero} by $A_h^{-1/2}$ and then applying the maximal estimate of Lemma \ref{lemma:fully DMR} gives 
\begin{equation*}
    \begin{aligned}
        \max_{0\le n\le N} \|A_h^{-1/2} \chi_n^{(0)}\|^p_{(X_{0,h}^q,X_{1,h}^q)_{1-1/p,p}} & \le C \sum_{n=0}^{N-1} \int_{t_n}^{t_{n+1}} \| A_h^{1/2} (u_h(t_{n+1})-u_h(t))\|_{X^q_0}^p \,dt 
        \\
        & \quad + C \sum_{n=0}^{N-1}  \int_{t_n}^{t_{n+1}} \|A_h^{-1/2} (U_n^h -u_h(t))\|_{X^q_0}^p \,dt 
    \end{aligned}
\end{equation*}
Taking expectations, using the temporal increment estimate of Step 2, the uniform boundedness of the inverse $A_h^{-1/2}$, and  \eqref{ineq: almost done} from Step 5, we obtain that 
\begin{equation*}
    \E  \max_{0\le n\le N} \|A_h^{-1/2} \chi_n^{(0)}\|^p_{(X_{0,h}^q,X_{1,h}^q)_{1-1/p,p}} \le C \tau^{p/2}.
\end{equation*}
Moreover, one has 
\begin{equation}\label{isomorphism between real int discrete spaces}
       A^{-1/2}_h \colon (X_{0,h}^q,X_{1,h}^q)_{1/2-1/p,p} \to (X_{0,h}^q,X_{1,h}^q)_{1-1/p,p} \text{ is an isomorphism}
   \end{equation}
   with constants that do not depend on $h$  (cf.\,\eqref{isomorphism between real int spaces}), and thus 
\begin{equation}\label{ineq:chi1}
    \E  \max_{0\le n\le N} \| \chi_n^{(0)}\|^p_{(X_{0,h}^q,X_{1,h}^q)_{1/2-1/p,p}} \le C \tau^{p/2}.
\end{equation}

Applying the maximal estimate of Lemma \ref{lemma:fully DMR}
to \eqref{eq:error component one}, together with the continuous embedding
$(X_{0,h}^q,X_{1,h}^q)_{1-1/p,p}
\hookrightarrow
(X_{0,h}^q,X_{1,h}^q)_{1/2-1/p,p}$, gives
\begin{equation*}
        \max_{0\le n\le N} \| \chi_n^{(1)}\|^p_{(X_{0,h}^q,X_{1,h}^q)_{1/2-1/p,p}} \le C \sum_{n=0}^{N-1}  \int_{t_n}^{t_{n+1}} \|P_h (u_h^3(t)-(U_{n+1}^h)^3) \|_{X^q_0}^p \, dt
\end{equation*}
We write
\[
    (U_{n+1}^h)^3-u_h^3(t)
    =
    \bigl((U_{n+1}^h)^3-u_h^3(t_{n+1})\bigr)
    +\bigl(u_h^3(t_{n+1})-u_h^3(t)\bigr).
\]
Consequently, using $e_{n+1}=U_{n+1}^h-u_h(t_{n+1})$, we obtain
\begin{equation}\label{ineq:cubic error bound}
\begin{aligned}
    \|(U_{n+1}^h)^3-u_h^3(t)\|_{L^q}
    &\le C 
    \|e_{n+1}\|_{L^q}
    \bigl(
        \|U_{n+1}^h\|_{L^\infty}^2
        +\|u_h(t_{n+1})\|_{L^\infty}^2
    \bigr)
    \\
    & \quad +C \|u_h^3(t_{n+1})-u_h^3(t)\|_{L^q}.
\end{aligned}
\end{equation}
Using the stability of $P_h$, the  $L^\infty$-stability estimates for $U_n^h$ and $u_h$ from Steps 4 and 1 respectively, the convergence estimate \eqref{ineq: almost done} for $e_n$ from Step 5, and the temporal increment estimate of Step 2, we infer that 
\begin{equation}\label{ineq:chi2}
    \E  \max_{0\le n\le N} \| \chi_n^{(1)}\|^p_{(X_{0,h}^q,X_{1,h}^q)_{1/2-1/p,p}} \le C \tau^{p/2}.
\end{equation}

Finally, we apply the maximal estimate of Lemma \ref{lemma:fully DSMR} to \eqref{eq:error component two} and obtain 
\begin{equation*}
    \E \max_{0\le n\le N} \| \chi_n^{(2)}\|^p_{(X_{0,h}^q,X_{1,h}^q)_{1/2-1/p,p}} \le C \E \left [ \sum_{n=0}^{N-1}  \int_{t_n}^{t_{n+1}} \|P_h (G(U_n^h)-G(u_h(t))) \|_{\gamma(\ell^2,X^q_0)}^p \,dt\right].
\end{equation*}
By the stability of $P_h$, the ideal property of $\gamma$-radonifying operators, and the Lipschitz estimate \eqref{ineq:AC noise estimates} of $G$, 
 $$\sum_{n=0}^{N-1}  \int_{t_n}^{t_{n+1}} \|P_h (G(U_n^h)-G(u_h(t))) \|_{\gamma(\ell^2,X^q_0)}^p \, dt  \le C  \sum_{n=0}^{N-1}  \int_{t_n}^{t_{n+1}} \|U_n^h-u_h(t) \|_{X^q_0}^p \, dt. $$
 Writing $U_n^h-u_h(t) = e_n +  (u_h(t_n)-u_h(t)) $ and using \eqref{ineq: almost done} from Step 5 and the temporal increment estimate of Step 2, we conclude that 
 \begin{equation}\label{ineq:chi3}
    \E  \max_{0\le n\le N} \| \chi_n^{(2)}\|^p_{(X_{0,h}^q,X_{1,h}^q)_{1/2-1/p,p}} \le C \tau^{p/2}.
\end{equation}

Combining \eqref{eq:error decomposition} with
\eqref{ineq:chi1}, \eqref{ineq:chi2}, and \eqref{ineq:chi3} proves
\eqref{ineq:stronger estimate}.

\end{proof}

\newpage
\bibliographystyle{plain}

\bibliography{refs}

\end{document}